\RequirePackage{fix-cm}
\documentclass[smallextended]{svjour3}       
\smartqed  

\usepackage{amsmath, amssymb, mathtools, graphicx, enumerate, verbatim}
\usepackage{cleveref}

\usepackage{tikz}
\usetikzlibrary{arrows.meta}
\usepackage{pgfplots}
\pgfplotsset{compat=1.17}
\usepgfplotslibrary{fillbetween}
\usetikzlibrary{patterns}

\DeclareMathOperator{\sgn}{sgn}
\DeclareMathOperator{\tr}{tr}
\DeclareMathOperator{\divergence}{div}
\begin{document}

\title{Oscillations in planar deficiency-one mass-action systems
\thanks{BB was supported by the Austrian Science Fund (FWF), project P32532.}
}


\author{Bal\'azs Boros \and Josef Hofbauer}

\institute{
B. Boros \at
Faculty of Mathematics, University of Vienna, Austria\\
\email{balazs.boros@univie.ac.at}
\and
J. Hofbauer \at
Faculty of Mathematics, University of Vienna, Austria\\
\email{josef.hofbauer@univie.ac.at}
}

\date{Received: date / Accepted: date}

\maketitle

\begin{abstract}
Whereas the positive equilibrium of a mass-action system with deficiency zero is always globally stable, for deficiency-one networks there are many different scenarios, mainly involving oscillatory behaviour. We present several examples, with  centers or multiple limit cycles.
\keywords{limit cycles \and centers \and Li\'enard systems \and reversible systems \and  mass-action kinetics \and deficiency one}
\end{abstract}

\section{Introduction}

In this paper we study mass-action systems in dimension two  with a unique positive equilibrium and deficiency one. The deficiency is a non-negative integer associated to any chemical reaction network, and is explained in \Cref{sec:planar_mass_action}. If the deficiency is zero, the Deficiency-Zero Theorem gives a rather complete picture: existence and uniqueness of a positive equilibrium is nicely characterized through the underlying directed graph of the chemical network. And if it exists, it is globally asymptotically stable (at least in dimension two). In contrast, the Deficiency-One Theorem is a purely static statement: there is at most one positive equilibrium, and if it exists, it is regular.
However, nothing is said about the dynamic behaviour. In the present paper we take a step towards filling this gap, at least in dimension two.

We show that Andronov--Hopf bifurcations may occur, even generalized ones, that produce more than one limit cycle near the equilibrium, and also degenerate ones that produce centers.
Note that the uniqueness and regularity of the positive equilibrium rules out the classical fixed point bifurcations such as saddle--node and pitchfork bifurcations.

We start with a brief summary of reaction network theory in \Cref{sec:planar_mass_action}. Then in \Cref{sec:quadrangle} we study the simplest weakly reversible network with deficiency one, an irreversible cycle along a quadrangle. This is always permanent, but we present examples with up to three limit cycles. On the other hand, we also give a sufficient condition that guarantees global stability of the positive equilibrium for all rate constants.

In \Cref{sec:chain} we consider irreversible chains of three reactions. Even though this is not weakly reversible, a positive equilibrium may exist, and could be globally stable, or may be surrounded by up to three limit cycles. Furthermore, the equilibrium may be surrounded by a continuum of closed orbits and a homoclinic orbit.

In \Cref{sec:three_reactions} we study three separate reactions, and produce two types of centers, and an example with four limit cycles.

Finally, in \Cref{sec:zigzag} we give a simple example where existence of the positive equilibrium depends on the rate constants.

\section{Planar mass-action systems}
\label{sec:planar_mass_action}

In this section we briefly introduce mass-action systems and related notions that are necessary for our exposition. We restrict to the case of two species. For more details about mass-action systems, consult e.g. \cite{feinberg:1987}, \cite{gunawardena:2003}. The symbol $\mathbb{R}_+$ denotes the set of positive real numbers.

\begin{definition}
A {\emph{planar Euclidean embedded graph}} (or a \emph{planar reaction network}) is a directed graph $(V,E)$, where $V$ is a nonempty finite subset of $\mathbb{R}^2$.
\end{definition}

Denote by $(a_1,b_1)$, $(a_2,b_2), \ldots, (a_m,b_m)$ the elements of $V$, and by $\mathsf{X}$ and $\mathsf{Y}$ the two \emph{species}. Accordingly, we often refer to $(a_i,b_i)$ as $a_i\mathsf{X}+b_i\mathsf{Y}$. We assume throughout that the \emph{reaction vectors} $(a_j-a_i,b_j-b_i)\in \mathbb{R}^2$ ($(i,j)\in E$) span $\mathbb{R}^2$. The concentrations of the species $\mathsf{X}$ and $\mathsf{Y}$ at time $\tau$ are denoted by $x(\tau)$ and $y(\tau)$, respectively.

\begin{definition}
A {\emph{planar mass-action system}} is a triple $(V,E, \kappa)$, where $(V,E)$ is a reaction network and $\kappa\colon E \to \mathbb{R}_+$ is the collection of the \emph{rate constants}. Its \emph{associated differential equation} on $\mathbb{R}^2_+$ is
\begin{align}
\label{eq:mass_action_ode}
\begin{split}
\dot{x} &= \sum_{(i,j)\in E} (a_j-a_i) \kappa_{ij} x^{a_i} y^{b_i},\\    
\dot{y} &= \sum_{(i,j)\in E} (b_j-b_i) \kappa_{ij} x^{a_i} y^{b_i}.    
\end{split}
\end{align}
\end{definition}

We remark that the translation of a network by $(\alpha,\beta)\in\mathbb{R}^2$ (i.e., taking $(a_i+\alpha,b_i+\beta)$ instead of $(a_i,b_i)$ for $i = 1,2, \dots, m$) amounts to multiplying the differential equation \eqref{eq:mass_action_ode} by the monomial $x^\alpha y^\beta$, an operation that does not have any effect on the main qualitative properties. Thus, any behaviour shown in this paper can also be realized with $a_i, b_i\geq 0$ for all $i=1,2,\ldots,m$, a setting that is more standard in the literature.

In some cases, a network property alone has consequences on the qualitative behaviour of the differential equation \eqref{eq:mass_action_ode}. For instance, weak reversibility implies permanence \cite[Theorem 4.6]{craciun:nazarov:pantea:2013}. We now define these terms.

\begin{definition}
A planar mass-action system $(V,E,\kappa)$ is \emph{weakly reversible} if every edge in $E$ is part of a directed cycle.
\end{definition}

\begin{definition}
A planar mass-action system is \emph{permanent} if there exists a compact set $K\subseteq\mathbb{R}^2_+$ with the property that for each solution $\tau \mapsto (x(\tau),y(\tau))$ with $(x(0),y(0))\in\mathbb{R}^2_+$ there exists a $\tau_0\geq0$ such that $(x(\tau),y(\tau))\in K$ holds for all $\tau\geq\tau_0$.
\end{definition}

\begin{theorem}
\label{thm:wr_permanence}
Weakly reversible planar mass-action systems are permanent.
\end{theorem}

We now recall two classical theorems on the number of positive equilibria for mass-action systems with low deficiency. The \emph{deficiency} of a planar reaction network $(V,E)$ is the non-negative integer $\delta = m - \ell - 2$, where $m = \lvert V \rvert$ and $\ell$ is the number of connected components of the directed graph $(V,E)$.

\begin{theorem}[Deficiency-Zero Theorem \cite{feinberg:1972}, \cite{horn:1972}, \cite{horn:jackson:1972}]
\label{thm:dfc0}
Assume that the deficiency of a planar mass-action system is zero. Then the following statements hold.
\begin{enumerate}[(i)]
\item There is no periodic solution that lies entirely in $\mathbb{R}^2_+$.
\item If the underlying network is weakly reversible then there exists a unique positive equilibrium. Furthermore, it is asymptotically stable.
\item If the underlying network is not weakly reversible then there is no positive equilibrium.
\end{enumerate}
\end{theorem}

Notice that the combination of \Cref{thm:dfc0,thm:wr_permanence} yields that the unique positive equilibrium of a weakly reversible deficiency-zero planar mass-action system is in fact globally asymptotically stable.

For stating the second classical result, we need one more term. For a directed graph $(V,E)$, denote by $t$ the number of its absorbing strong components.

\begin{theorem}[Deficiency-One Theorem \cite{feinberg:1995}]
\label{thm:dfc1}
Assume that the deficiency of a planar mass-action system is one. Further, assume that $\ell=t=1$. Then the following statements hold.
\begin{enumerate}[(i)]
\item If the underlying network is weakly reversible then there exists a unique positive equilibrium.
\item If the underlying network is not weakly reversible then the number of positive equilibria is either $0$ or $1$.
\item The determinant of the Jacobian matrix at a positive equilibrium is nonzero.
\end{enumerate}
\end{theorem}

We now highlight the main differences between the conclusions of the above two theorems. For a planar mass-action system that falls under the assumptions of the Deficiency-One Theorem,
\begin{enumerate}[(A)]

\item in case the underlying network is not weakly reversible,

\begin{enumerate}[(a)]
\item a positive equilibrium can nevertheless exist,
\item whether there exists a positive equilibrium might depend on the specific values of the rate constants,
\item even if there exists a unique positive equilibrium, there could be unbounded solutions as well as solutions that approach the boundary of $\mathbb{R}^2_+$,
\end{enumerate}

\item regardless of weak reversibility,
\begin{enumerate}[(a)]
\item the unique positive equilibrium could be unstable (however, the Jacobian matrix there is guaranteed to be nonsingular),
\item there is no information about the existence of periodic solutions.
\end{enumerate}

\end{enumerate}

Points (a) and (b) in (A) above are studied in detail in \cite{boros:2012} and \cite{boros:2013b}, respectively. In this paper, we touch these questions only briefly: we show a network in \Cref{sec:zigzag} for which the existence of a positive equilibrium is dependent on the specific choice of the rate constants.

Investigation of points (a) and (b) in (B) above is the main motivation for the present paper. The only (published) example so far of a reaction network that satisfies the Deficiency-One Theorem but with an unstable positive equilibrium (and presumably a limit cycle) seems to be the three-species network
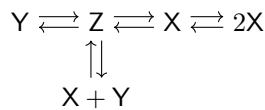
\begin{center}
\begin{tikzpicture}

\node (P1) at (0,0)  {$\mathsf{Y}$};
\node (P2) at (1,0)  {$\mathsf{Z}$};
\node (P3) at (2,0)  {$\mathsf{X}$};
\node (P4) at (3,0)  {$2\mathsf{X}$};
\node (P5) at (1,-1) {$\mathsf{X} + \mathsf{Y}$};

\draw[arrows={->[scale=.8]},transform canvas={yshift=2pt}] (P1) to node {} (P2);
\draw[arrows={->[scale=.8]},transform canvas={yshift=-2pt}] (P2) to node {} (P1);
\draw[arrows={->[scale=.8]},transform canvas={yshift=2pt}] (P2) to node {} (P3);
\draw[arrows={->[scale=.8]},transform canvas={yshift=-2pt}] (P3) to node {} (P2);
\draw[arrows={->[scale=.8]},transform canvas={yshift=2pt}] (P3) to node {} (P4);
\draw[arrows={->[scale=.8]},transform canvas={yshift=-2pt}] (P4) to node {} (P3);
\draw[arrows={->[scale=.8]},transform canvas={xshift=2pt}] (P2) to node {} (P5);
\draw[arrows={->[scale=.8]},transform canvas={xshift=-2pt}] (P5) to node {} (P2);

\end{tikzpicture}
\end{center}
which is due to Feinberg \cite[(4.12)]{feinberg:1995}. In this paper we show that such examples are abundant already for two species. Note that Feinberg's example is a bimolecular one. It is shown in 
\cite{pota:1983}
that the only bimolecular two species system with periodic solutions is the Lotka reaction \cite{lotka:1920}.

\Cref{sec:quadrangle,sec:chain} are devoted to studying the cycle of four irreversible reactions and the chain of three irreversible reactions, respectively. In \Cref{sec:three_reactions} we examine a generalization of the latter: three irreversible reactions that do not necessarily form a chain. The networks in \Cref{sec:quadrangle,sec:chain,sec:zigzag} all satisfy the assumptions of the Deficiency-One Theorem.

Finally, we remark that the Deficiency-One Theorem has a version for the case of more than one connected component (i.e., $\ell\geq2$). However, this does not cover the deficiency-one networks in \Cref{sec:three_reactions}. Nevertheless, the conclusions $(ii)$ and $(iii)$ in \Cref{thm:dfc1} hold.

\section{Quadrangle}
\label{sec:quadrangle}

In this section we study the mass-action system
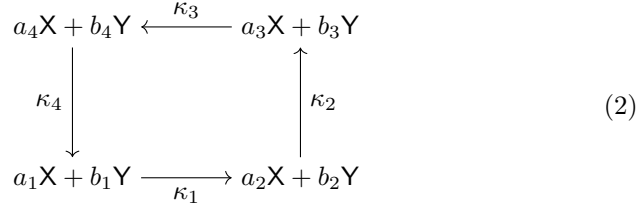
\begin{align}
\label{eq:quadrangle_network}
\begin{split}
\begin{tikzpicture}[scale=3]

\node (P1) at (0,0)  {$a_1 \mathsf{X} + b_1 \mathsf{Y}$};
\node (P2) at (1,0)  {$a_2 \mathsf{X} + b_2 \mathsf{Y}$};
\node (P3) at (1,2/3)  {$a_3 \mathsf{X} + b_3 \mathsf{Y}$};
\node (P4) at (0,2/3)  {$a_4 \mathsf{X} + b_4 \mathsf{Y}$};

\draw[arrows={->[scale=1]},below] (P1) to node {$\kappa_1$} (P2);
\draw[arrows={->[scale=1]},right] (P2) to node {$\kappa_2$} (P3);
\draw[arrows={->[scale=1]},above] (P3) to node {$\kappa_3$} (P4);
\draw[arrows={->[scale=1]},left] (P4) to node {$\kappa_4$} (P1);

\end{tikzpicture}
\end{split}
\end{align}
and its associated differential equation
\begin{align}
\label{eq:quadrangle_ode}
\begin{split}
\dot{x} &= (a_2-a_1)\kappa_1 x^{a_1}y^{b_1} + (a_3-a_2)\kappa_2 x^{a_2}y^{b_2} +\\ & \quad  +(a_4-a_3)\kappa_3 x^{a_3}y^{b_3} + (a_1-a_4)\kappa_4 x^{a_4}y^{b_4}, \\
\dot{y} &= (b_2-b_1)\kappa_1 x^{a_1}y^{b_1} + (b_3-b_2)\kappa_2 x^{a_2}y^{b_2} +\\ & \quad  +(b_4-b_3)\kappa_3 x^{a_3}y^{b_3} + (b_1-b_4)\kappa_4 x^{a_4}y^{b_4}
\end{split}
\end{align}
under the non-degeneracy assumption that $(a_1,b_1)$, $(a_2,b_2)$, $(a_3,b_3)$, $(a_4,b_4)$ are distinct and do not lie on a line.

Note that the mass-action system \eqref{eq:quadrangle_network} is weakly reversible and its deficiency is $\delta = 4 - 1 - 2 = 1$. By the Deficiency-One Theorem \cite[Theorem 4.2]{feinberg:1995}, there exists a unique positive equilibrium for the differential equation \eqref{eq:quadrangle_ode}. Moreover, the determinant of the Jacobian matrix at the equilibrium does not vanish \cite[Theorem 4.3]{feinberg:1995}. Since additionally the system is permanent \cite[Theorem 4.6]{craciun:nazarov:pantea:2013}, the index of the equilibrium is $+1$ \cite[Theorem 19.3]{hofbauer:sigmund:1988}. 
Hence, the determinant is positive, and consequently, the unique positive equilibrium is asymptotically stable (respectively, unstable) if the trace is negative (respectively, positive). In case the trace is zero, the eigenvalues are purely imaginary and some further work is required to decide stability of the equilibrium.

In \Cref{subsec:quadrangle_parallelogram}, we present a system for which the unique positive equilibrium is unstable and a stable limit cycle exists. In \Cref{subsec:quadrangle_three_limit_cycles}, we prove that even three limit cycles are possible for the differential equation \eqref{eq:quadrangle_ode}. Finally, in \Cref{subsec:quadrangle_dulac}, we describe a subclass of the quadrangle networks \eqref{eq:quadrangle_network} that are globally stable for all rate constants.

\subsection{Unstable equilibrium and a stable limit cycle}
\label{subsec:quadrangle_parallelogram}

Let us consider the mass-action system \eqref{eq:quadrangle_network} with
\begin{align*}
(a_1,b_1) = (0,1), (a_2,b_2) = (1,0), (a_3,b_3) = (1,2), (a_4,b_4) = (0,3).
\end{align*}
Thus, the network and its associated differential equation take the form
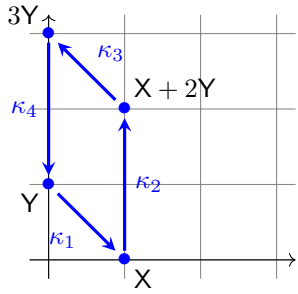
\begin{center}
\begin{tikzpicture}

\draw [step=1, gray, very thin] (-0.25,-0.25) grid (3.25,3.25);
\draw [ ->, black] (-0.25,0)--(3.25,0);
\draw [ ->, black] (0,-0.25)--(0,3.25);

\node[inner sep=0,outer sep=1] (P1) at (0,1) {\large \textcolor{blue}{$\bullet$}};
\node[inner sep=0,outer sep=1] (P2) at (1,0) {\large \textcolor{blue}{$\bullet$}};
\node[inner sep=0,outer sep=1] (P3) at (1,2) {\large \textcolor{blue}{$\bullet$}};
\node[inner sep=0,outer sep=1] (P4) at (0,3) {\large \textcolor{blue}{$\bullet$}};

\node [below left]  at (P1) {$\mathsf{Y}$};
\node [below right] at (P2) {$\mathsf{X}$};
\node [above right] at (P3) {$\mathsf{X}+2\mathsf{Y}$};
\node [above left]  at (P4) {$3\mathsf{Y}$};

\draw[arrows={-stealth},very thick,blue,below left]  (P1) to node {$\kappa_1$} (P2);
\draw[arrows={-stealth},very thick,blue,right]       (P2) to node {$\kappa_2$} (P3);
\draw[arrows={-stealth},very thick,blue,above right] (P3) to node {$\kappa_3$} (P4);
\draw[arrows={-stealth},very thick,blue,left]        (P4) to node {$\kappa_4$} (P1);

\node [] at (4,3/2) {and};

\node [] at (7.5,3/2) {$\begin{aligned}
\dot{x} &= \kappa_1 y - \kappa_3 xy^2, \\
\dot{y} &= -\kappa_1y + 2 \kappa_2 x + \kappa_3 xy^2 - 2\kappa_4 y^3.
\end{aligned}$};

\end{tikzpicture}
\end{center}
A short calculation shows that the unique positive equilibrium is given by 
\begin{align*}
(\overline{x},\overline{y})=\left(\left(\frac{\kappa_1^3 \kappa_4}{\kappa_3^3 \kappa_2}\right)^{\frac{1}{4}},\left(\frac{\kappa_1 \kappa_2}{\kappa_3 \kappa_4}\right)^{\frac{1}{4}}\right)
\end{align*}
and the trace of the Jacobian matrix at the equilibrium is positive if and only if
\begin{align*}
\frac{\kappa_1}{\kappa_2} > \left(6\sqrt{\frac{\kappa_3}{\kappa_4}}+\sqrt{\frac{\kappa_4}{\kappa_3}}\right)^2.
\end{align*}
By picking rate constants that make the trace positive, one gets a system, where the positive equilibrium is repelling, and, by combining permanence and the Poincar\'e--Bendixson Theorem, there must exist a stable limit cycle.

\subsection{Three limit cycles}
\label{subsec:quadrangle_three_limit_cycles}

Let us consider the mass-action system \eqref{eq:quadrangle_network} with
\begin{align*}
(a_1,b_1) = (0,1), (a_2,b_2) = (0,0), (a_3,b_3) = (1,2), (a_4,b_4) = (1,5).
\end{align*}
Thus, the network and its associated differential equation take the form
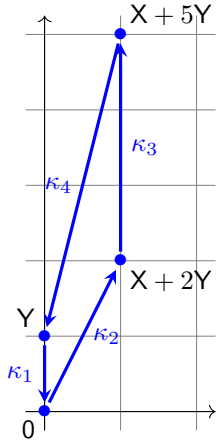
\begin{center}
\begin{tikzpicture}

\draw [step=1, gray, very thin] (-0.25,-0.25) grid (2.25,5.25);
\draw [ ->, black] (-0.25,0)--(2.25,0);
\draw [ ->, black] (0,-0.25)--(0,5.25);

\node[inner sep=0,outer sep=1] (P1) at (0,1) {\large \textcolor{blue}{$\bullet$}};
\node[inner sep=0,outer sep=1] (P2) at (0,0) {\large \textcolor{blue}{$\bullet$}};
\node[inner sep=0,outer sep=1] (P3) at (1,2) {\large \textcolor{blue}{$\bullet$}};
\node[inner sep=0,outer sep=1] (P4) at (1,5) {\large \textcolor{blue}{$\bullet$}};

\node [above left]  at (P1) {$\mathsf{Y}$};
\node [below left] at (P2) {$\mathsf{0}$};
\node [below right] at (P3) {$\mathsf{X}+2\mathsf{Y}$};
\node [above right]  at (P4) {$\mathsf{X}+5\mathsf{Y}$};

\draw[arrows={-stealth},very thick,blue,left]  (P1) to node {$\kappa_1$} (P2);
\draw[arrows={-stealth},very thick,blue,right]       (P2) to node {$\kappa_2$} (P3);
\draw[arrows={-stealth},very thick,blue,right] (P3) to node {$\kappa_3$} (P4);
\draw[arrows={-stealth},very thick,blue,left]        (P4) to node {$\kappa_4$} (P1);

\node [] at (3,5/2) {and};

\node [] at (6.5,5/2) {$\begin{aligned}
\dot{x} &=               \kappa_2                  - \kappa_4 xy^5, \\
\dot{y} &= -\kappa_1y + 2\kappa_2 + 3\kappa_3 xy^2 - 4\kappa_4 xy^5.
\end{aligned}$};

\end{tikzpicture}
\end{center}
Our goal is to show that there exist rate constants $\kappa_1$, $\kappa_2$, $\kappa_3$, $\kappa_4$ such that the above differential equation has three limit cycles.

Linear scaling of the differential equation by the equilibrium $(\overline{x},\overline{y})$, followed by a multiplication by $\overline{x}$ yields
\begin{align}
\label{eq:quadrangle_three_limit_cycles_ode}
\begin{split}
\dot{x} &= \overline{\kappa}_2 - \overline{\kappa}_4 x y^5, \\
\dot{y} &= K[-\overline{\kappa}_1 y + 2\overline{\kappa}_2 + 3\overline{\kappa}_3 x y^2-4\overline{\kappa}_4 x y^5],
\end{split}
\end{align}
where
\begin{align*}
\overline{\kappa}_1 = \kappa_1 \overline{y}, \overline{\kappa}_2 = \kappa_2, \overline{\kappa}_3 = \kappa_3 \overline{x} \overline{y}^2, \overline{\kappa}_4 = \kappa_4 \overline{x} \overline{y}^5,
\text{ and } K=\frac{\overline{x}}{\overline{y}}.
\end{align*}
As a result of the scaling, the positive equilibrium is moved to $(1,1)$, and the focal value computations become somewhat more convenient. Note that
\begin{align*}
0 &= \overline{\kappa}_2 - \overline{\kappa}_4, \\
0 &= -\overline{\kappa}_1 + 2\overline{\kappa}_2 + 3\overline{\kappa}_3 - 4\overline{\kappa}_4.
\end{align*}
From this, we obtain that
\begin{align*}
\overline{\kappa}_1 &= \overline{\kappa}_4 \gamma,\\
\overline{\kappa}_2 &= \overline{\kappa}_4,\\
\overline{\kappa}_3 &= \overline{\kappa}_4 \frac{\gamma+2}{3}
\end{align*}
for some $\gamma>0$. After dividing by $\overline{\kappa}_4$, the differential equation \eqref{eq:quadrangle_three_limit_cycles_ode} thus becomes
\begin{align}
\label{eq:quadrangle_three_limit_cycles_ode_K}
\begin{split}
\dot{x} &= 1 - x y^5, \\
\dot{y} &= K[-\gamma y + 2 + (\gamma+2) x y^2-4 x y^5],
\end{split}
\end{align}
where $K>0$ and $\gamma>0$. One finds that the trace of the Jacobian matrix at the equilibrium $(1,1)$ vanishes for $\gamma=16+\frac{1}{K}$. Under this, the first focal value is
\begin{align*}
L_1 = \frac{\pi(3416 K^3 + 1250 K^2 - 29 K - 5)}{20\sqrt{2(2+35K)^3}},
\end{align*}
which is zero for $K=K_0\approx 0.06862$, negative for $0<K<K_0$, and positive for $K>K_0$. Assuming $K=K_0$, one finds that the second focal value, $L_2$, is approximately $0.01293$, a positive number.

Take now $K=K_0$ and $\gamma=16+\frac{1}{K_0}$. Since the first nonzero focal value is positive, the equilibrium $(1,1)$ is repelling. First, perturb $K$ to a slightly smaller value, and simultaneously perturb $\gamma$ in order to maintain the relation $\gamma=16+\frac{1}{K}$. Then $L_1<0$, and thus the equilibrium $(1,1)$ becomes asymptotically stable, and an unstable limit cycle $\Gamma_1$ is created. Next perturb $\gamma$ to a slightly larger value. Then the trace becomes positive, and thus the equilibrium $(1,1)$ becomes unstable again, and a stable limit cycle $\Gamma_0$ is created. Finally, by the permanence of the system, the Poincar\'e-Bendixson Theorem guarantees that a stable limit cycle surrounds $\Gamma_1$. Therefore, we have shown that there exist $K>0$ and $\gamma>0$ such that the differential equation \eqref{eq:quadrangle_three_limit_cycles_ode_K} has at least three limit cycles.

We conclude this subsection by a remark. By keeping $b_4>2$ a parameter (instead of fixing its value to $5$), one could find parameter values for which $L_1=0$, $L_2=0$, $L_3<0$ holds (with $b_4 \approx 4.757$ and $K \approx 0.0909$). Then one can bifurcate three small limit cycles from the equilibrium.

\subsection{Global stability of the equilibrium}
\label{subsec:quadrangle_dulac}

As we have seen in \Cref{subsec:quadrangle_parallelogram,subsec:quadrangle_three_limit_cycles}, the unique positive equilibrium of the differential equation \eqref{eq:quadrangle_ode} could be unstable for some rate constants. However, under a certain condition on the relative position of the four points $(a_1,b_1)$, $(a_2,b_2)$, $(a_3,b_3)$, $(a_4,b_4)$, one can conclude global asymptotic stability of the unique positive equilibrium for all rate constants.

The differential equation \eqref{eq:quadrangle_ode} is permanent and has a unique positive equilibrium. Furthermore, the determinant of the Jacobian matrix is positive there. Hence, by the Poincar\'e-Bendixson Theorem, global asymptotic stability of the equilibrium is equivalent to the non-existence of a periodic solution. One can preclude the existence of a periodic solution by the Bendixson-Dulac test: if there exists a function $h\colon \mathbb{R}^2_+ \to \mathbb{R}_+$ such that $\divergence(hf,hg)<0$ then the differential equation
\begin{align*}
\dot{x} &= f(x,y),\\
\dot{y} &= g(x,y)
\end{align*}
cannot have a periodic solution that lies entirely in $\mathbb{R}^2_+$.

With $f(x,y)$ and $g(x,y)$ denoting the r.h.s. of the equations for $\dot{x}$ and $\dot{y}$ in \eqref{eq:quadrangle_ode}, respectively, and taking $h(x,y) = x^{-\alpha} y^{-\beta}$, one finds
\begin{align*}
\frac{\divergence(hf,hg)}{h}(x,y)&=\sum_{i=1}^4(\alpha-a_i)(a_{i}-a_{i+1})\kappa_i x^{a_i-1}y^{b_i}+\\
&\quad +\sum_{i=1}^4(\beta-b_i)(b_{i}-b_{i+1})\kappa_i x^{a_i}y^{b_i-1},
\end{align*}
where $a_5=a_1$ and $b_5=b_1$ by convention. Ignoring the degenerate case $a_1=a_2=a_3=a_4$, one finds that $(\alpha-a_i)(a_i-a_{i+1})\leq0$ for each $i=1,2,3,4$ and $(\alpha-a_i)(a_i-a_{i+1})<0$ for some $i=1,2,3,4$ if
\begin{align*}
a_1\leq a_2 \leq a_3 \leq a_4 &\text{ and } a_3 \leq \alpha \leq a_4,\text{ or }\\
a_1\leq a_2 \leq a_4 \leq a_3 &\text{ and } a_2 \leq \alpha \leq a_4,\text{ or }\\
a_1\leq a_3 \leq a_2 \leq a_4 &\text{ and } a_3 \leq \alpha \leq a_2,\text{ or }\\
a_1\leq a_3 \leq a_4 \leq a_2 &\text{ and } a_3 \leq \alpha \leq a_4,\text{ or }\\
a_1\leq a_4 \leq a_3 \leq a_2 &\text{ and } a_1 \leq \alpha \leq a_4,\text{ or }\\
a_1 =   a_4 \leq a_2 \leq a_3 &\text{ and } a_2 \leq \alpha \leq a_3,\text{ or }\\
a_1\leq a_4 \leq a_2 =    a_3 &\text{ and } a_1 \leq \alpha \leq a_4,\text{ or }\\
a_1\leq a_4 =    a_2 \leq a_3 &\text{ and } \alpha = a_2.
\end{align*}
On the other hand, if $a_1 < a_4 < a_2 < a_3$ then no matter how one fixes $\alpha$, at least one of $(\alpha-a_2)(a_2-a_3)$ and $(\alpha-a_4)(a_4-a_1)$ is positive. Notice that we covered all configurations with $a_1=\min(a_1,a_2,a_3,a_4)$. All the other cases are treated similarly. Also, it works analogously with the $b_j$'s and $\beta$.

\begin{proposition}
\label{prop:quadrangle_dulac}
Consider the differential equation \eqref{eq:quadrangle_ode} and let the indices $i$ and $j$ satisfy $a_i=\min(a_1,a_2,a_3,a_4)$ and $b_j=\min(b_1,b_2,b_3,b_4)$, respectively. Assume that both $a_i<a_{i+3}<a_{i+1}<a_{i+2}$ and $b_j<b_{j+3}<b_{j+1}<b_{j+2}$ are violated (where $a_5=a_1$, $a_6=a_2$, $a_7=a_3$ and $b_5=b_1$, $b_6=b_2$, $b_7=b_3$ by convention). Then there is no periodic solution and the unique positive equilibrium is globally asymptotically stable.
\end{proposition}
\begin{proof}
By the above discussion, one can find $\alpha$ and $\beta$ such that after multiplying by $h(x,y)=x^{-\alpha}y^{-\beta}$, the r.h.s. of the differential equation \eqref{eq:quadrangle_ode} has negative divergence everywhere. Then, by the Bendixson-Dulac test, there is no periodic solution and therefore the unique positive equilibrium is globally asymptotically stable. \qed
\end{proof}

In other words, if there exists a periodic solution then at least one of $a_i<a_{i+3}<a_{i+1}<a_{i+2}$ and $b_j<b_{j+3}<b_{j+1}<b_{j+2}$ in \Cref{prop:quadrangle_dulac} holds. The index-free way to express $a_i<a_{i+3}<a_{i+1}<a_{i+2}$ and $b_j<b_{j+3}<b_{j+1}<b_{j+2}$ is to say that the projection of the quadrangle to a horizontal line and a vertical line, respectively, take the form
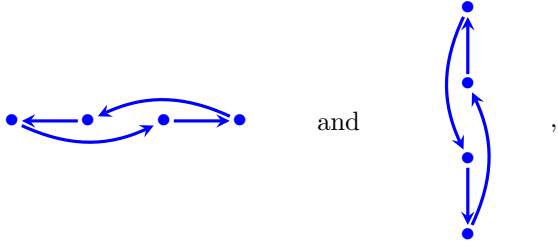
\begin{center}
\begin{tikzpicture}

\node[inner sep=0,outer sep=1] (P1) at (0,0) {\large \textcolor{blue}{$\bullet$}};
\node[inner sep=0,outer sep=1] (P2) at (2,0) {\large \textcolor{blue}{$\bullet$}};
\node[inner sep=0,outer sep=1] (P3) at (3,0) {\large \textcolor{blue}{$\bullet$}};
\node[inner sep=0,outer sep=1] (P4) at (1,0) {\large \textcolor{blue}{$\bullet$}};

\draw[arrows={-stealth},very thick,blue,bend right=25] (P1) to node {} (P2);
\draw[arrows={-stealth},very thick,blue] (P2) to node {} (P3);
\draw[arrows={-stealth},very thick,blue,bend right=25] (P3) to node {} (P4);
\draw[arrows={-stealth},very thick,blue] (P4) to node {} (P1);

\node [] at (4.3,0) {and};

\node[inner sep=0,outer sep=1] (P1) at (6,-1.5) {\large \textcolor{blue}{$\bullet$}};
\node[inner sep=0,outer sep=1] (P2) at (6, 0.5) {\large \textcolor{blue}{$\bullet$}};
\node[inner sep=0,outer sep=1] (P3) at (6, 1.5) {\large \textcolor{blue}{$\bullet$}};
\node[inner sep=0,outer sep=1] (P4) at (6,-0.5) {\large \textcolor{blue}{$\bullet$}};

\draw[arrows={-stealth},very thick,blue,bend right=25] (P1) to node {} (P2);
\draw[arrows={-stealth},very thick,blue] (P2) to node {} (P3);
\draw[arrows={-stealth},very thick,blue,bend right=25] (P3) to node {} (P4);
\draw[arrows={-stealth},very thick,blue] (P4) to node {} (P1);

\node [] at (7,0) {\phantom{A},};

\end{tikzpicture}
\end{center}
respectively, where some arrows are bent in order to avoid overlapping.

Finally, since the mass-action systems in \Cref{subsec:quadrangle_parallelogram,subsec:quadrangle_three_limit_cycles} have a periodic solution for some rate constants, at least one of $a_i<a_{i+3}<a_{i+1}<a_{i+2}$ and $b_j<b_{j+3}<b_{j+1}<b_{j+2}$ must hold. Indeed, in each subsection $b_j<b_{j+3}<b_{j+1}<b_{j+2}$ holds with $j=2$.

\section{Chain of three reactions}
\label{sec:chain}

In this section we study the mass-action system
\begin{align}
\label{eq:chain_network}
\begin{split}
a_1 \mathsf{X} + b_1 \mathsf{Y} \stackrel{\kappa_1}{\longrightarrow} a_2 \mathsf{X} + b_2 \mathsf{Y} \stackrel{\kappa_2}{\longrightarrow} a_3 \mathsf{X} + b_3 \mathsf{Y}\stackrel{\kappa_3}{\longrightarrow} a_4 \mathsf{X} + b_4 \mathsf{Y}
\end{split}
\end{align}
and its associated differential equation
\begin{align}
\label{eq:chain_ode}
\begin{split}
\dot{x} &= (a_2-a_1)\kappa_1 x^{a_1}y^{b_1} + (a_3-a_2)\kappa_2 x^{a_2}y^{b_2} + (a_4-a_3)\kappa_3 x^{a_3}y^{b_3}, \\
\dot{y} &= (b_2-b_1)\kappa_1 x^{a_1}y^{b_1} + (b_3-b_2)\kappa_2 x^{a_2}y^{b_2} + (b_4-b_3)\kappa_3 x^{a_3}y^{b_3}
\end{split}
\end{align}
under the non-degeneracy assumption that
\begin{align}
\label{eq:chain_nondeg}
(a_1,b_1), (a_2,b_2), (a_3,b_3) \text{ do not lie on a line.}
\end{align}

By the Deficiency-One Theorem, the number of positive equilibria for the differential equation \eqref{eq:chain_ode} is either $0$ or $1$. Our first goal is to understand when is it $0$ and when is it $1$. Crucial for this is the relative position of the four points $P_i=(a_i,b_i)$ for $i = 1, 2, 3, 4$ in the plane. Define the numbers $h_1$, $h_2$, $h_3$, $h_4$ by
\begin{align*}
\begin{split}
h_1 &= \Delta(243), \\
h_2 &= \Delta(134), \\
h_3 &= \Delta(142), \\
h_4 &= \Delta(123),
\end{split}
\end{align*}
where $\Delta(ijk)=\det(P_j-P_i, P_k-P_i)$ is twice the signed area of the triangle $P_iP_jP_k$. The quantity $\Delta(ijk)$ is thus positive (respectively, negative) if the sequence $P_i$, $P_j$, $P_k$, $P_i$ of points are positively (respectively, negatively) oriented. The quantity $\Delta(ijk)$ is zero if the three points $P_i$, $P_j$, $P_k$ lie on a line. Note also that
\begin{align*}
\Delta(ijk) = \Delta(jki) =  \Delta(kij) = -\Delta(jik) = -\Delta(ikj) = -\Delta(kji)
\end{align*}
and $h_1 + h_2 + h_3 + h_4 = 0$.

Denote by $f(x,y)$ and $g(x,y)$ the r.h.s. of the equations for $\dot{x}$ and $\dot{y}$ in \eqref{eq:chain_ode}, respectively. By taking
\begin{align*}
(b_3-b_2)f(x,y)-(a_3-a_2)g(x,y)=0,\\
(b_4-b_3)f(x,y)-(a_4-a_3)g(x,y)=0,
\end{align*}
one obtains after a short calculation that the equilibrium equations take the form
\begin{align}
\label{eq:chain_binomial}
\begin{split}
(h_1+h_2+h_3) \kappa_1 x^{a_1}y^{b_1} &= h_1 \kappa_3 x^{a_3}y^{b_3},\\
(h_1+h_2) \kappa_1 x^{a_1}y^{b_1} &= h_1 \kappa_2 x^{a_2}y^{b_2}.
\end{split}
\end{align}
Thus, if there exists a positive equilibrium, $h_1$, $h_1+h_2$, $h_1+h_2+h_3$ must all have the same sign. If all of them are zero then $P_1$, $P_2$, $P_3$, $P_4$ lie on a line, contradicting the non-degeneracy assumption \eqref{eq:chain_nondeg}. If the common sign is nonzero then in particular $h_4=-(h_1+h_2+h_3)\neq0$, so $P_1$, $P_2$, $P_3$ do not lie on a line, and thus the obtained binomial equation \eqref{eq:chain_binomial} has exactly one positive solution for each choice of the rate constants. Let us stress that the existence of a positive equilibrium does not depend on the specific choice of the rate constants.

Next, we discuss the geometric meaning of
\begin{align*}
\sgn(h_1) = \sgn(h_1+h_2) = \sgn(h_1+h_2+h_3) \neq 0.
\end{align*}
Assume that
\begin{align*}
h_1<0, h_1+h_2<0, h_1+h_2+h_3<0.
\end{align*}
Since $\Delta(234)=-\Delta(243)=-h_1>0$, the sequence $P_2$, $P_3$, $P_4$, $P_2$ is oriented counterclockwise. Similarly, since $\Delta(231)=\Delta(123)=h_4=-(h_1+h_2+h_3)>0$, the sequence $P_2$, $P_3$, $P_1$, $P_2$ is oriented counterclockwise, too. Thus, $P_1$ and $P_4$ lie on the same side of the line through $P_2$ and $P_3$ (the green open half-plane in the left panel in \Cref{fig:chain_exist_equilibrium} shows where $P_4$ can be located for $h_1<0$ to hold). Since additionally $h_1+h_2<0$ holds, $P_1$ and $P_4$ lie on the same side of the line that is through $P_3$ and is parallel to the line through $P_1$ and $P_2$ (the red open half-plane in the left panel in \Cref{fig:chain_exist_equilibrium} shows where $P_4$ can be located for $h_1+h_2<0$ to hold). This latter follows from the fact that $h_1+h_2=\det(P_2-P_1,P_4-P_3)$. In other words, the sum of the angles $\measuredangle P_1P_2P_3$ and $\measuredangle P_2P_3P_4$ is smaller than $180^\circ$ (see the two red arcs in right panel in \Cref{fig:chain_exist_equilibrium}). The case $h_1>0$, $h_1+h_2>0$, $h_1+h_2+h_3>0$ is treated similarly, and we obtain the following result.

\begin{proposition}
\label{prop:chain_exist_equilibrium}
Consider the differential equation \eqref{eq:chain_ode}. Then the following four statements are equivalent.
\begin{enumerate}[(a)]
\item There exists a positive equilibrium.
\item There exists a unique positive equilibrium.
\item $\sgn(h_1) = \sgn(h_1+h_2) = \sgn(h_1+h_2+h_3) \neq 0$
\item The points $P_1$ and $P_4$ lie on the same side of the line through $P_2$ and $P_3$, and additionally $\measuredangle P_1P_2P_3 + \measuredangle P_2P_3P_4 < 180^\circ$.
\end{enumerate}
In particular, the existence of a positive equilibrium is independent of the values of $\kappa_1$, $\kappa_2$, $\kappa_3$.
\end{proposition}

We remark that the equivalence of $(a)$, $(b)$, and $(c)$ in \Cref{prop:chain_exist_equilibrium} also follows from \cite[Corollaries 4.6 and 4.7]{boros:2013b}, where the existence of a positive equilibrium is discussed for general deficiency-one mass-action systems.

\begin{figure}[ht]
\begin{center}
\begin{tabular}{cc}
\begin{tikzpicture}[scale=0.8]
\begin{axis}[axis lines=none, axis equal, no marks, xmin=-5, xmax=5, ymin=-4, ymax=6, samples=3]
\addplot+[black, ultra thin, name path=A] {x/2};
\addplot+[name path=B] {x/2+10};
\addplot[pattern=vertical lines, pattern color=black!20!green] fill between [of=A and B];

\addplot+[black, ultra thin, name path=C] {-x+3};
\addplot+[name path=D] {-x-10};
\addplot[pattern=horizontal lines, pattern color=red!30!] fill between [of=C and D];

\addplot+[black, ultra thin] {-x-3/2};

\node (P1) at (-4,5/2)  {\large \textcolor{blue}{$\bullet$}};
\node [right] at (P1) {$(a_1,b_1)$};
\node (P2) at (-1,-1/2)  {\large \textcolor{blue}{$\bullet$}};
\node [right] at (P2) {$(a_2,b_2)$};
\node (P3) at (2,1)  {\large \textcolor{blue}{$\bullet$}};
\node [right] at (P3) {$(a_3,b_3)$};

\end{axis}

\node[right] at (0.5,6.3) {half-plane with green vertical lines:};
\node[right] at (0.5,6) {location of $P_4$ for $h_1<0$ to hold};

\node[right] at (0.5,-0.2) {half-plane with red horizontal lines:};
\node[right] at (0.5,-0.5) {location of $P_4$ for $h_1+h_2<0$ to hold};

\end{tikzpicture} &
\begin{tikzpicture}[scale=.5]

\node[inner sep=0,outer sep=1] (P1) at (-4,5/2)  {\large \textcolor{blue}{$\bullet$}};
\node[inner sep=0,outer sep=1] (P2) at (-1,-1/2) {\large \textcolor{blue}{$\bullet$}};
\node[inner sep=0,outer sep=1] (P3) at (2,1) {\large \textcolor{blue}{$\bullet$}};
\node[inner sep=0,outer sep=1] (P4) at (-2,3)  {\large \textcolor{blue}{$\bullet$}};

\node [above] at (P1) {$(a_1,b_1)$};
\node [below] at (P2) {$(a_2,b_2)$};
\node [right] at (P3) {$(a_3,b_3)$};
\node [above] at (P4) {$(a_4,b_4)$};

\draw[red] (0,0) arc (26.5651:135:1.118);
\draw[red] (1,3/2) arc (153.4349:206.5651:1.118);

\draw[arrows={-stealth},very thick,blue] (P1) to node {} (P2);
\draw[arrows={-stealth},very thick,blue] (P2) to node {} (P3);
\draw[arrows={-stealth},very thick,blue] (P3) to node {} (P4);

\node at (0,-4.5) {};

\end{tikzpicture}
\end{tabular}
\end{center}
\caption{For a positive equilibrium to exist, the point $P_4$ is located in the sector that is the intersection of the green and red open half-spaces (left panel). Equivalently, the sum of the two angles indicated is less than $180^\circ$ (right panel).}
\label{fig:chain_exist_equilibrium}
\end{figure}
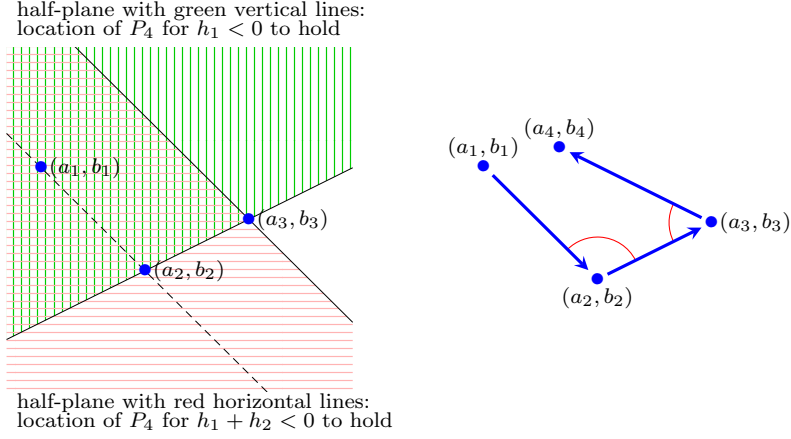

Now that we understand when the mass-action system \eqref{eq:chain_network} has a positive equilibrium, our next goal is to find parameter values for which the equilibrium is surrounded by three limit cycles (\Cref{subsec:chain_three_limit_cycles}) or by a continuum of closed orbits (\Cref{subsec:chain_reversible_center}). We prepare for these by moving the equilibrium to $(1,1)$.

Linear scaling of the differential equation \eqref{eq:chain_ode} by the equilibrium $(\overline{x},\overline{y})$, followed by a multiplication by $\overline{x}$ yields
\begin{align}
\label{eq:chain_ode_scaled}
\begin{split}
\dot{x} &= (a_2-a_1)\overline{\kappa}_1 x^{a_1}y^{b_1} + (a_3-a_2)\overline{\kappa}_2 x^{a_2}y^{b_2} + (a_4-a_3)\overline{\kappa}_3 x^{a_3}y^{b_3}, \\
\dot{y} &= K[(b_2-b_1)\overline{\kappa}_1 x^{a_1}y^{b_1} + (b_3-b_2)\overline{\kappa}_2 x^{a_2}y^{b_2} + (b_4-b_3)\overline{\kappa}_3 x^{a_3}y^{b_3}],
\end{split}
\end{align}
where
\begin{align*}
\overline{\kappa}_1 = \kappa_1 \overline{x}^{a_1} \overline{y}^{b_1}, \overline{\kappa}_2 = \kappa_2 \overline{x}^{a_2} \overline{y}^{b_2}, 
\overline{\kappa}_3 = \kappa_3 \overline{x}^{a_3} \overline{y}^{b_3}, 
\text{ and } K=\frac{\overline{x}}{\overline{y}}.
\end{align*}
As a result of the scaling, the positive equilibrium is moved to $(1,1)$. Further, it follows by \eqref{eq:chain_binomial} that
\begin{align}
\label{eq:chain_kappasubst}
\begin{split}
\overline{\kappa}_1 &= \lambda h_1, \\
\overline{\kappa}_2 &= \lambda (h_1+h_2), \\
\overline{\kappa}_3 &= \lambda (h_1+h_2+h_3)
\end{split}
\end{align}
for some $\lambda\neq0$, which is positive (respectively, negative) if $h_1$, $h_1+h_2$, $h_1+h_2+h_3$ are all positive (respectively, negative).

Denote by $J$ the Jacobian matrix of \eqref{eq:chain_ode_scaled} at the equilibrium $(1,1)$. A short calculation shows that
\begin{align*}
\det J = \frac{h_1+h_2+h_3}{\lambda}K\overline{\kappa}_1\overline{\kappa}_2\overline{\kappa}_3,
\end{align*}
and thus, $\det J>0$.

\subsection{Three limit cycles}
\label{subsec:chain_three_limit_cycles}

Let us consider now the mass-action system \eqref{eq:chain_network} with
\begin{align*}
(a_1,b_1) = (0,0), (a_2,b_2) = (0,-q), (a_3,b_3) = \left(1,\frac{1}{2}\right), (a_4,b_4) = \left(0,\frac{1}{2}+r\right),
\end{align*}
where $q>0$ and $r>0$, i.e., take the mass-action system
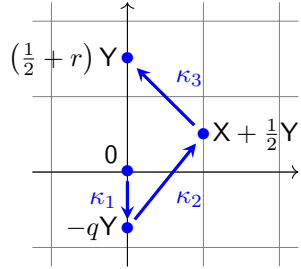
\begin{center}
\begin{tikzpicture}

\draw [step=1, gray, very thin] (-1.25,-1.25) grid (2.25,2.25);
\draw [ ->, black] (-1.25,0)--(2.25,0);
\draw [ ->, black] (0,-1.25)--(0,2.25);

\node[inner sep=0,outer sep=1] (P1) at (0,0) {\large \textcolor{blue}{$\bullet$}};
\node[inner sep=0,outer sep=1] (P2) at (0,-3/4) {\large \textcolor{blue}{$\bullet$}};
\node[inner sep=0,outer sep=1] (P3) at (1,1/2) {\large \textcolor{blue}{$\bullet$}};
\node[inner sep=0,outer sep=1] (P4) at (0,3/2) {\large \textcolor{blue}{$\bullet$}};

\node [above left]  at (P1) {$\mathsf{0}$};
\node [left] at (P2)  {$-q\mathsf{Y}$};
\node [right] at (P3) {$\mathsf{X}+\frac{1}{2}\mathsf{Y}$};
\node [left]  at (P4) {$\left(\frac{1}{2}+r\right)\mathsf{Y}$};

\draw[arrows={-stealth},very thick,blue,left]  (P1) to node {$\kappa_1$} (P2);
\draw[arrows={-stealth},very thick,blue,below right] (P2) to node {$\kappa_2$} (P3);
\draw[arrows={-stealth},very thick,blue,above right] (P3) to node {$\kappa_3$} (P4);

\end{tikzpicture}
\end{center}
Then $h_1=-\left(q+r+\frac{1}{2}\right)$, $h_2=r+\frac{1}{2}$, $h_3=0$, and therefore $h_1$, $h_1+h_2$, $h_1+h_2+h_3$ are all negative, so we take $\lambda$ to be negative. Taking $\lambda=-\frac{1}{q}$ in \eqref{eq:chain_kappasubst}, the associated scaled differential equation \eqref{eq:chain_ode_scaled} takes the form
\begin{align}
\label{eq:chain_ode_scaled_Krq}
\begin{split}
\dot{x} &= y^{-q} - xy^{\frac{1}{2}}, \\
\dot{y} &= K\left[-\left(q+r+\frac{1}{2}\right) + \left(q+\frac{1}{2}\right) y^{-q} + r xy^\frac{1}{2}\right].
\end{split}
\end{align}

Next we prove that there exist $q>0$, $r>0$, $K>0$ such that $\tr J=L_1=L_2=0$ and $L_3<0$, where $L_i$ is the $i$th focal value at the equilibrium $(1,1)$.

\begin{proposition}
\label{prop:chain_L1_L2_L3}
Consider the differential equation \eqref{eq:chain_ode_scaled_Krq}. Then there exist $q>0$, $r>0$, $K>0$ such that $\tr J=L_1=L_2=0$ and $L_3<0$.
\end{proposition}
\begin{proof}
Since $\tr J = -1+\frac{r-q(2q+1)}{2}K$, the trace vanishes with $K=\frac{2}{r-q(2q+1)}$ for $q>0$ and $r>q(2q+1)$. Under this, one obtains that
\begin{align*}
L_1 = \frac{\pi r [3r(1-2q)-q(4q^2+16q+7)]}{8(2q+1)[r-q(2q+1)]^\frac{3}{2}\sqrt{2q(q+r+1/2)}}.
\end{align*} 
Taking also into account that $q>0$ and $r>q(2q+1)$, one obtains that $L_1=0$ if and only if $0<q<\frac{1}{2}$ and $r=\frac{q(4q^2+16q+7)}{3(1-2q)}$. Under this, one obtains that
\begin{align*}
L_2 = \frac{\pi (2q+7)^2(3-2q)(4q-1)\sqrt{2q+3}}{1536(2q+1)^4}.
\end{align*}
Taking also into account that $0<q<\frac{1}{2}$, one obtains that $L_2=0$ if and only if $q=\frac{1}{4}$. With this, one computes $L_3$ and gets $L_3=-\frac{625\pi}{110592}\sqrt{\frac{7}{2}}$.

The parameter value for which $\tr J=L_1=L_2=0$ and $L_3<0$ hold are obtained by substitution. This yields $q=\frac{1}{4}$, $r=\frac{15}{8}$, $K=\frac{4}{3}$. \qed
\end{proof}

\begin{corollary}
Consider the differential equation \eqref{eq:chain_ode_scaled_Krq}. Then there exist $q>0$, $r>0$, $K>0$ such that $(1,1)$ is unstable and is surrounded by $3$ limit cycles ($2$ stable and $1$ unstable).
\end{corollary}
\begin{proof}
Take $q=\frac{1}{4}$, $r=\frac{15}{8}$, $K=\frac{4}{3}$. As we saw in the proof of \Cref{prop:chain_L1_L2_L3}, then $\tr J = L_1 = L_2 = 0$ and $L_3<0$. Since the first nonzero focal value is negative, the equilibrium $(1,1)$ is asymptotically stable.

First, perturb $q$ to a slightly larger value, and simultaneously perturb $r$ and $K$ in order to maintain the relations $r=\frac{q(4q^2+16q+7)}{3(1-2q)}$ and $K=\frac{2}{r-q(2q+1)}$. Then $L_2>0$, and thus the equilibrium $(1,1)$ becomes unstable, and a stable limit cycle $\Gamma_2$ is created.

Next, perturb $r$ to a slightly smaller value, and simultaneously perturb $K$ in order to maintain the relation $K=\frac{2}{r-q(2q+1)}$. Then $L_1<0$, and thus the equilibrium $(1,1)$ becomes asymptotically stable, and an unstable limit cycle $\Gamma_1$ is created.

Finally, perturb $K$ to a slightly larger value. Then $\tr J>0$, and thus the equilibrium $(1,1)$ becomes unstable, and a stable limit cycle $\Gamma_0$ is created. \qed
\end{proof}

We remark (without proving) that the mass-action systems of this subsection are permanent for all $q>0$ and $r>0$. In particular, the ones with at least three limit cycles are permanent.

\subsection{Reversible center}
\label{subsec:chain_reversible_center}

Let us consider now the mass-action system \eqref{eq:chain_network} with
\begin{align}
\label{eq:chain_reversible_abpq}
(a_1,b_1) = (0,0), (a_2,b_2) = (p,q), (a_3,b_3) = (q,p), (a_4,b_4) = \left(q-p,p+\frac{q^2}{p}\right),
\end{align}
where $pq<0$ and $p+q\neq 0$. We will prove that the unique positive equilibrium of this mass-action system is a center, provided the rate constants $\kappa_1$, $\kappa_2$, $\kappa_3$ are set appropriately.

By taking $\lambda=-\frac{1}{p^2-q^2}$ in \eqref{eq:chain_kappasubst}, we have $\overline{\kappa}_1=\frac{p-q}{p}$, $\overline{\kappa}_2=-\frac{q}{p-q}$, $\overline{\kappa}_3=1$, which are indeed all positive under the assumptions on $p$ and $q$. Setting $K=-\frac{p}{q}$, the associated scaled differential equation \eqref{eq:chain_ode_scaled} then takes the form
\begin{align}
\label{eq:chain_ode_scaled_reversible}
\begin{split}
\dot{x} &= (p-q)+q x^p y^q - p x^q y^p, \\
\dot{y} &= (q-p)+p x^p y^q - q x^q y^p.
\end{split}
\end{align}

\begin{proposition}
The equilibrium $(1,1)$ is a center of the differential equation \eqref{eq:chain_ode_scaled_reversible}, provided $pq<0$ and $p+q\neq 0$ hold.
\end{proposition}
\begin{proof}
Note that the Jacobian matrix at $(1,1)$ equals to $(p^2-q^2)\begin{pmatrix*}[r]0 & -1\\ 1 & 0 \end{pmatrix*}$. Thus, the eigenvalues are purely imaginary. Since the differential equation \eqref{eq:chain_ode_scaled_reversible} is of the form
\begin{align*}
\dot{x}&=f(x,y),\\
\dot{y}&=-f(y,x),
\end{align*}
the system is reversible w.r.t. the line $x=y$ and $(1,1)$ is indeed a center. \qed
\end{proof}

We depicted the typical phase portraits in \Cref{fig:chain_homoclinic}. The one for $p+q>0$ suggests that the closed orbits are surrounded by a homoclinic orbit at the origin. In \Cref{prop:chain_homoclinic} we show that this is indeed the case.

\begin{figure}[ht]
\begin{center}
\input{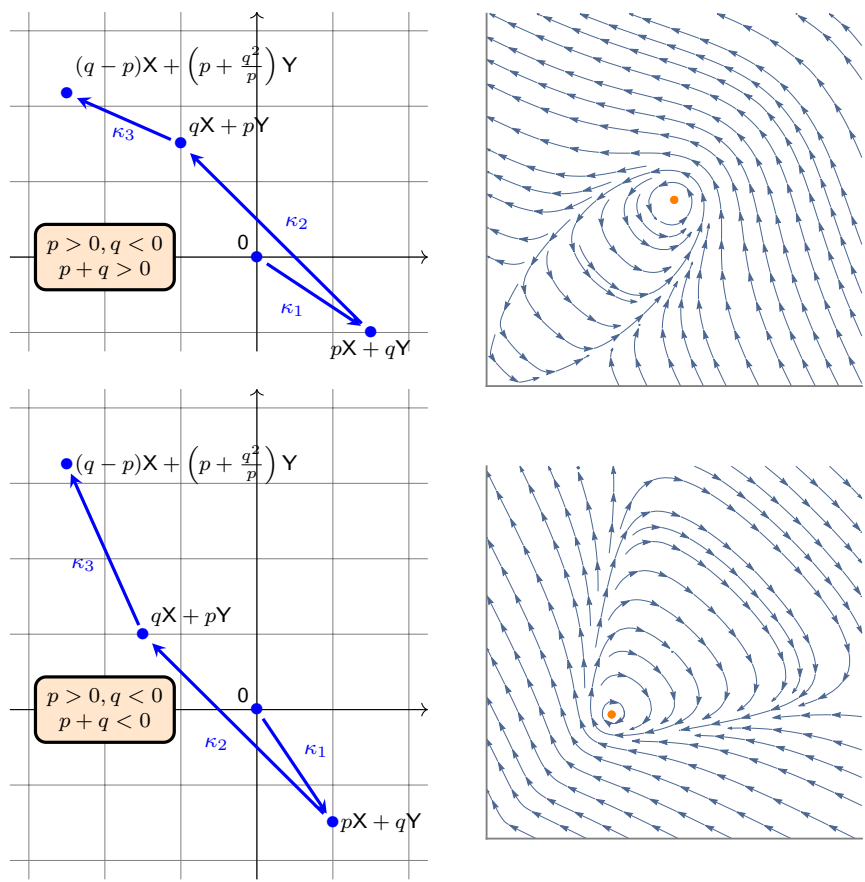}
\end{center}
\caption{The mass-action systems \eqref{eq:chain_ode} with the substitution \eqref{eq:chain_reversible_abpq} (left column), and the phase portraits of the corresponding scaled differential equation \eqref{eq:chain_ode_scaled_reversible} (right column). The top row is for $p+q>0$, while the bottom row is for $p+q<0$. Notice that the union of closed orbits is bounded for $p+q>0$, and unbounded for $p+q<0$.}
\label{fig:chain_homoclinic}
\end{figure}

\begin{proposition}
\label{prop:chain_homoclinic}
Consider the differential equation \eqref{eq:chain_ode_scaled_reversible} with $p > 0$, $q < 0$, and $p+q>0$. Then the region consisting of closed orbits is bounded. Furthermore, all closed orbits lie inside a homoclinic orbit, whose $\alpha$-- and $\omega$--limit is the origin.
\end{proposition}
\begin{proof} 
At the rightmost point of a closed orbit of \eqref{eq:chain_ode_scaled_reversible}, $\dot x = 0$ and $x > y > 0$ hold. We show that for $L$ sufficiently large, $\dot x < 0$ holds on the vertical line segment $\{(x,y)\mid x=L, 0 < y < L\}$. Indeed,
\begin{align*}
\dot{x} = p-q + qL^p y^q - p L^q y ^p < p-q + q L^{p+q} < 0 \text{ for } L > \left(1 - \frac{p}{q}\right)^{\frac1{p+q}},  
\end{align*}
where we used $p>0$, $q<0$, and $p+q>0$. By symmetry/reversibility, all closed orbits are contained in the square $[0,L]^2$.

We next show that there is an invariant curve asymptotic to the toric ray $x^p y^q = \frac{p-q}{p}$ at the origin on which the flow goes  away from the origin.  
Let $z = x^{-p}y^{-q}$ and rewrite (or \lq\lq blow up\rq\rq) the system \eqref{eq:chain_ode_scaled_reversible} in terms of $(x,z)$. Using $y = x^{-\frac{p}{q}} z^{-\frac1{q}}$, we obtain
\begin{align*}
\dot x &= p-q + qz^{-1} - px^{\frac{q^2 - p^2}{q}} z^{-\frac{p}{q}},\\
\dot z &= -pz \frac{\dot x}{x} - qz\frac{\dot y}{y} = 
 -qx^{\frac{p}{q}} z^{\frac1{q}} [ (q-p)z + p + \cdots ],
\end{align*}
where $\cdots$ stands for four more terms of higher order near $x = 0$, i.e., with $x$ having a positive exponent. After we multiply this transformed system $(\dot x, \dot z)$ by $x^{-\frac{p}{q}} z^{-\frac1{q}}$, we obtain a system that is smooth on the non-negative quadrant $\mathbb{R}^2_{\geq0}$, with the $z$-axis being invariant. On the $z$-axis we have
\begin{align*}
\dot z = -q [p+ (q-p)z]
\end{align*}
with an equilibrium at $(\widehat{x},\widehat{z}) = \left(0,\frac{p}{p-q}\right)$. Since $-\frac{p}{q} > 1$, $\frac{q^2 - p^2}{q}> 0$, and
\begin{align*}
\dot x &= x^{-\frac{p}{q}} z^{-\frac1{q}}\left(p-q + qz^{-1} - px^{\frac{q^2 - p^2}{q}} z^{-\frac{p}{q}}\right),
\end{align*}
the eigenvalue at $(\widehat{x}, \widehat{z})$, transverse to the $z$-axis is zero. Near the equilibrium $(\widehat{x}, \widehat{z})$,
\begin{align*}
\dot x \approx x^{-\frac{p}{q}} {\widehat{z}}^{-\frac1{q}}[p-q + q\widehat{z}^{-1}] = x^{-\frac{p}{q}} {\widehat{z}}^{-\frac1{q}}\frac{p^2-q^2}{p} > 0.
\end{align*}
Therefore, the flow on the center manifold goes in the positive $x$--direction. Transforming this center manifold back produces the promised invariant curve along the toric ray $x^p y^q = \frac{p-q}{p}$.  

Since $\dot y > 0$ near the $x$--axis and $\dot x < 0$ on the vertical line segment $(L,y)$ for large $L$, this invariant curve has to cross the line $y=x$. 
By symmetry/reversibility, the mirror image connects back to the origin, following the toric ray  $x^q y^p = \frac{p-q}{p}$ when approaching the origin as $\tau \to +\infty$. Therefore, this curve is a homoclinic orbit.
\qed 
\end{proof}

\section{Three reactions}
\label{sec:three_reactions}

In this section we study the mass-action system
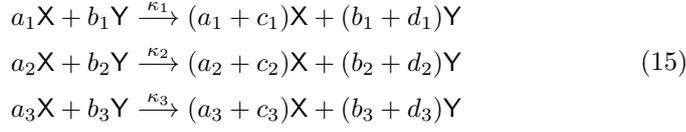
\begin{align}
\label{eq:three_reactions_network}
\begin{split}
a_1 \mathsf{X} + b_1 \mathsf{Y} &\stackrel{\kappa_1}{\longrightarrow} (a_1+c_1) \mathsf{X} + (b_1+d_1) \mathsf{Y} \\
a_2 \mathsf{X} + b_2 \mathsf{Y} &\stackrel{\kappa_2}{\longrightarrow} (a_2+c_2) \mathsf{X} + (b_2+d_2) \mathsf{Y} \\
a_3 \mathsf{X} + b_3 \mathsf{Y} &\stackrel{\kappa_3}{\longrightarrow} (a_3+c_3) \mathsf{X} + (b_3+d_3) \mathsf{Y}
\end{split}
\end{align}
and its associated differential equation
\begin{align}
\label{eq:three_reactions_ode}
\begin{split}
\dot{x} &= c_1\kappa_1 x^{a_1}y^{b_1} + c_2\kappa_2 x^{a_2}y^{b_2} + c_3\kappa_3 x^{a_3}y^{b_3}, \\
\dot{y} &= d_1\kappa_1 x^{a_1}y^{b_1} + d_2\kappa_2 x^{a_2}y^{b_2} + d_3\kappa_3 x^{a_3}y^{b_3}
\end{split}
\end{align}
under the non-degeneracy assumptions that
\begin{align}
\label{eq:three_reactions_nondeg}
\begin{split}
&(a_1,b_1), (a_2,b_2), (a_3,b_3) \text{ do not lie on a line,}\\
&\text{none of }(c_1,d_1), (c_2,d_2), (c_3,d_3)\text{ equals to }(0,0),\text{ and}\\
&(c_1,d_1), (c_2,d_2), (c_3,d_3)\text{ span }\mathbb{R}^2.
\end{split}    
\end{align}

Our first goal is to understand the number of positive equilibria. We find that $(\overline{x},\overline{y})\in\mathbb{R}^2_+$ is an equilibrium if and only if
\begin{align}
\label{eq:three_reactions_binomial}
\begin{split}
(c_1d_2-c_2d_1)\kappa_1 \overline{x}^{a_1} \overline{y}^{b_1}&=(c_2d_3-c_3d_2)\kappa_3 \overline{x}^{a_3} \overline{y}^{b_3},\\
(c_3d_1-c_1d_3)\kappa_1 \overline{x}^{a_1} \overline{y}^{b_1}&=(c_2d_3-c_3d_2)\kappa_2 \overline{x}^{a_2} \overline{y}^{b_2}.
\end{split}
\end{align}
Notice that, by the non-degeneracy assumptions \eqref{eq:three_reactions_nondeg}, the three numbers $c_2d_3-c_3d_2$, $c_3d_1-c_1d_3$, $c_1d_2-c_2d_1$ cannot all be zero. Thus, taking also into account that $(a_1,b_1)$, $(a_2,b_2)$, $(a_3,b_3)$ do not lie on a line, the existence of a positive equilibrium is equivalent to
\begin{align}
\label{eq:three_reactions_common_sign}
\sgn(c_2d_3-c_3d_2)=\sgn(c_3d_1-c_1d_3)=\sgn(c_1d_2-c_2d_1)\neq0.
\end{align}
Furthermore, once there exists a positive equilibrium, it is unique. Note also that whether there exists a positive equilibrium is independent of the choice of the rate constants $\kappa_1$, $\kappa_2$, $\kappa_3$.

Now that we understand when the mass-action system \eqref{eq:three_reactions_network} has a positive equilibrium, our next goal is to find parameter values for which the equilibrium is surrounded by four limit cycles (\Cref{subsec:three_reactions_four_limit_cycles}) or by a continuum of closed orbits (\Cref{subsec:three_reactions_reversible_center,subsec:three_reactions_lienard_center}). We remark that the center problem is solved in the special case when one of the reactions is vertical and another one is horizontal \cite{boros:hofbauer:mueller:regensburger:2017}. Further, the existence of two limit cycles is also discussed there.

We prepare for the rest of this section by moving the equilibrium to $(1,1)$. Linear scaling of the differential equation \eqref{eq:three_reactions_ode} by the equilibrium $(\overline{x},\overline{y})$, followed by a multiplication by $\overline{x}$ yields
\begin{align}
\label{eq:three_reactions_ode_scaled}
\begin{split}
\dot{x} &= c_1\overline{\kappa}_1 x^{a_1}y^{b_1} + c_2\overline{\kappa}_2 x^{a_2}y^{b_2} + c_3\overline{\kappa}_3 x^{a_3}y^{b_3}, \\
\dot{y} &= K(d_1\overline{\kappa}_1 x^{a_1}y^{b_1} + d_2\overline{\kappa}_2 x^{a_2}y^{b_2} + d_3\overline{\kappa}_3 x^{a_3}y^{b_3}),
\end{split}
\end{align}
where
\begin{align}
\label{eq:three_reactions_kappa_kappabar}
\overline{\kappa}_1 = \kappa_1 \overline{x}^{a_1} \overline{y}^{b_1}, \overline{\kappa}_2 = \kappa_2 \overline{x}^{a_2} \overline{y}^{b_2}, 
\overline{\kappa}_3 = \kappa_3 \overline{x}^{a_3} \overline{y}^{b_3}, 
\text{ and } K=\frac{\overline{x}}{\overline{y}}.
\end{align}
As a result of the scaling, the positive equilibrium is moved to $(1,1)$. Further, it follows by \eqref{eq:three_reactions_binomial} that
\begin{align}
\label{eq:three_reactions_kappasubst}
\begin{split}
\overline{\kappa}_1 &= \lambda(c_2d_3-c_3d_2), \\
\overline{\kappa}_2 &= \lambda(c_3d_1-c_1d_3), \\
\overline{\kappa}_3 &= \lambda(c_1d_2-c_2d_1)
\end{split}
\end{align}
for some $\lambda\neq0$, which is positive (respectively, negative) if the common sign in \eqref{eq:three_reactions_common_sign} is positive (respectively, negative).

Denote by $J$ the Jacobian matrix of \eqref{eq:three_reactions_ode_scaled} at the equilibrium $(1,1)$. A short calculation shows that
\begin{align}
\label{eq:three_reactions_detJ}
\det J = \frac{1}{\lambda}K \overline{\kappa}_1\overline{\kappa}_2\overline{\kappa}_3[a_1(b_2-b_3)+a_2(b_3-b_1)+a_3(b_1-b_2)].
\end{align}

\subsection{Four limit cycles}
\label{subsec:three_reactions_four_limit_cycles}

In this subsection we discuss why we strongly conjecture that there exist parameter values for which the differential equation \eqref{eq:three_reactions_ode} has at least $4$ limit cycles.

Let us consider now the mass-action system \eqref{eq:three_reactions_network} with
\begin{center}
\begin{tabular}{ccc}
$\begin{aligned}
(a_1,b_1) &= (0,0), \\
(c_1,d_1) &= (0,-1),
\end{aligned}$ & $\begin{aligned}
(a_2,b_2) &= (0,-1), \\
(c_2,d_2) &= (1,-1),
\end{aligned}$ & $\begin{aligned}
(a_3,b_3) &= (a,b), \\
(c_3,d_3) &= (-1,d)
\end{aligned}$
\end{tabular}
\end{center}
for $a>0$, $b>-1$, $d>0$ with $1+bd>0$, i.e., take the mass-action system
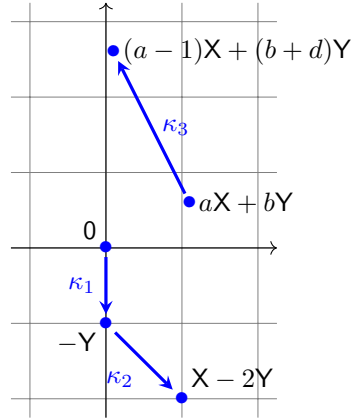
\begin{center}
\begin{tikzpicture}

\draw [step=1, gray, very thin] (-1.25,-2.25) grid (2.25,3.25);
\draw [ ->, black] (-1.25,0)--(2.25,0);
\draw [ ->, black] (0,-2.25)--(0,3.25);

\node[inner sep=0,outer sep=1] (P1) at (0,0) {\large \textcolor{blue}{$\bullet$}};
\node[inner sep=0,outer sep=1] (P2a) at (0,-1) {\large \textcolor{blue}{$\bullet$}};
\node[inner sep=0,outer sep=1] (P2b) at (1,-2) {\large \textcolor{blue}{$\bullet$}};
\node[inner sep=0,outer sep=1] (P3a) at (1.1,0.6) {\large \textcolor{blue}{$\bullet$}};
\node[inner sep=0,outer sep=1] (P3b) at (0.1,2.6) {\large \textcolor{blue}{$\bullet$}};

\node [above left]  at (P1) {$\mathsf{0}$};
\node [below left] at (P2a)  {$-\mathsf{Y}$};
\node [above right] at (P2b)  {$\mathsf{X}-2\mathsf{Y}$};
\node [right] at (P3a) {$a\mathsf{X}+b\mathsf{Y}$};
\node [right]  at (P3b) {$(a-1)\mathsf{X}+(b+d)\mathsf{Y}$};

\draw[arrows={-stealth},very thick,blue,left]  (P1) to node {$\kappa_1$} (P2a);
\draw[arrows={-stealth},very thick,blue,below left] (P2a) to node {$\kappa_2$} (P2b);
\draw[arrows={-stealth},very thick,blue,right] (P3a) to node {$\kappa_3$} (P3b);

\end{tikzpicture}
\end{center}
Then $\overline{\kappa}_1 = \lambda(d-1)$, $\overline{\kappa}_2 = \lambda$, $\overline{\kappa}_3 = \lambda$ in \eqref{eq:three_reactions_kappasubst}, and hence, the existence of a positive equilibrium is equivalent to $d>1$. Take $\lambda=1$. With these, the scaled differential equation takes the form
\begin{align*}
\dot{x} &= \frac{1}{y} - x^{a}y^{b}, \\
\dot{y} &= K\left(-(d-1) - \frac{1}{y} + d x^{a}y^{b}\right),
\end{align*}
and, by \eqref{eq:three_reactions_detJ}, $\det J = K(d-1)a$, which is positive, because $d>1$, $a>0$.

Set $K=\frac{a}{1+bd}$ to make $\tr J$ equal to zero. The first focal value, $L_1$, is then
\begin{align*}
L_1 = \frac{\pi a (2a^2d+a(1+bd)-(1+bd)^2)}{8a\sqrt{(d-1)(1+bd)^3}},
\end{align*}
which vanishes for $b=\frac{-2+a(1+\sqrt{1+8d})}{2d}$. After the elimination of $b$ by this, one finds that the second focal value, $L_2$, vanishes along a curve in the $(a,d)$--plane. That curve contains the points $\left(1,\frac{165}{49}\right)$ and $\left(\frac{1+\sqrt{3961}}{60},3\right)$, and it turns out the third focal value, $L_3$, is negative at the former point and positive at the latter one. Thus, there exist $a$, $b$, $d$ such that $L_1=L_2=L_3=0$ (namely, we numerically find that this happens at $a = \widehat{a} \approx 1.01282$, $b = \widehat{b} \approx 0.65463$, $d = \widehat{d} \approx 3.28862$). Since, again numerically, we see that $L_4$ is negative for $\widehat{a}$, $\widehat{b}$, $\widehat{d}$, we conjecture that there exist parameter values for which the unique positive equilibrium of the differential equation \eqref{eq:three_reactions_ode} is asymptotically stable, and is surrounded by four limit cycles $\Gamma_0$, $\Gamma_1$, $\Gamma_2$, $\Gamma_3$, which are unstable, stable, unstable, stable, respectively. Since the formulas for $L_2$, $L_3$, and $L_4$ get complicated, we cannot handle them analytically. This is why we leave the existence of four limit cycles a conjecture.

\subsection{Reversible center}
\label{subsec:three_reactions_reversible_center}

Let us consider now the mass-action system \eqref{eq:three_reactions_network} with 
\begin{align}
\label{eq:three_reactions_reversible_ab_pq}
(a_1,b_1)=(0,0),(a_2,b_2)=(p,q), (a_3,b_3)=(q,p),
\end{align}
assuming $\lvert p\rvert\neq\lvert q\rvert$. Its associated scaled differential equation is then
\begin{align}
\label{eq:three_reactions_ode_scaled_reversible}
\begin{split}
\dot{x} &= c_1\overline{\kappa}_1 + c_2\overline{\kappa}_2 x^py^q + c_3\overline{\kappa}_3 x^qy^p, \\
\dot{y} &= K(d_1\overline{\kappa}_1 + d_2\overline{\kappa}_2 x^py^q + d_3\overline{\kappa}_3 x^qy^p).
\end{split}
\end{align}

\begin{proposition}
\label{prop:three_reactions_reversible_center}
Consider the differential equation \eqref{eq:three_reactions_ode_scaled_reversible} with \eqref{eq:three_reactions_kappasubst}. Assume that $\frac{1}{\lambda}K(p^2-q^2)>0$ and
\begin{align}
\label{eq:three_reactions_reversibiliy_condition}
\begin{split}
c_1&=-K d_1,\\
c_2\overline{\kappa}_2&=-K d_3\overline{\kappa}_3,\\ c_3\overline{\kappa}_3&=-K d_2\overline{\kappa}_2
\end{split}
\end{align}
hold. Then the equilibrium $(1,1)$ is a center.
\end{proposition}
\begin{proof}
The determinant and the trace of the Jacobian matrix at $(1,1)$ are
\begin{align*}
\frac{1}{\lambda}K \overline{\kappa}_1\overline{\kappa}_2\overline{\kappa}_3(p^2-q^2)\text{ and }
p(c_2\overline{\kappa}_2+K d_3\overline{\kappa}_3)+q(c_3\overline{\kappa}_3+K d_2\overline{\kappa}_2),
\end{align*}
respectively. By the assumptions, the former is positive, the latter is zero, and therefore, the eigenvalues are purely imaginary. Since, by \eqref{eq:three_reactions_reversibiliy_condition}, the differential equation \eqref{eq:three_reactions_ode_scaled_reversible} is of the form
\begin{align*}
\dot{x}&=f(x,y),\\
\dot{y}&=-f(y,x),
\end{align*}
the system is reversible w.r.t. the line $x=y$ and $(1,1)$ is indeed a center. \qed
\end{proof}

\begin{corollary}
\label{cor:three_reactions_reversible_center}
Consider the differential equation \eqref{eq:three_reactions_ode_scaled_reversible} with \eqref{eq:three_reactions_kappasubst}. Assume that
\begin{enumerate}[(i)]
\item $p^2>q^2$,
\item $\sgn c_1 = -\sgn d_1 = \sgn d_2 = -\sgn c_3\neq0$,
\item $\sgn c_2=-\sgn d_3$,
\item $K=-\frac{c_1}{d_1}$
\end{enumerate}
hold. In case $\sgn c_2=-\sgn d_3\neq0$, require additionally that
\begin{align*}
\left\lvert\frac{d_3}{c_3}\right\rvert<\left\lvert\frac{d_1}{c_1}\right\rvert<\left\lvert\frac{d_2}{c_2}\right\rvert \text{ and }\left\lvert\frac{d_1}{c_1}\right\rvert = \sqrt{\left\lvert\frac{d_2}{c_2}\right\rvert \left\lvert\frac{d_3}{c_3}\right\rvert}.
\end{align*}
Then the equilibrium $(1,1)$ is a center.
\end{corollary}
\begin{proof}
In case $\sgn c_2=-\sgn d_3=0$, each of $c_2d_3-c_3d_2$, $c_3d_1-c_1d_3$, $c_1d_2-c_2d_1$ is positive (by $(ii)$), and hence $\lambda>0$. Further, $K>0$ (by $(ii)$ and $(iv)$). Taking also into account $(i)$, it follows that $\frac{1}{\lambda}K(p^2-q^2)>0$. Verification of \eqref{eq:three_reactions_reversibiliy_condition} is straightforward.

In case $\sgn c_2=-\sgn d_3\neq0$, either
\begin{align*}
&\sgn c_1 = -\sgn d_1 = -\sgn c_2 = \sgn d_2 = -\sgn c_3 = \sgn d_3\text{ and }\\
&\frac{d_2}{c_2}<\frac{d_1}{c_1}<\frac{d_3}{c_3}<0
\end{align*}
or
\begin{align*}
&\sgn c_1 = -\sgn d_1 = \sgn c_2 = \sgn d_2 = -\sgn c_3 = -\sgn d_3\text{ and }\\
&\frac{d_1}{c_1}<0<\frac{d_3}{c_3}<\frac{d_2}{c_2}.
\end{align*}
In each of these cases, one again obtains $\lambda>0$ and hence $\frac{1}{\lambda}K(p^2-q^2)>0$. By using the fact that $\left\lvert\frac{d_1}{c_1}\right\rvert$ is the geometric mean of $\left\lvert\frac{d_2}{c_2}\right\rvert$ and $\left\lvert\frac{d_3}{c_3}\right\rvert$, one readily checks  \eqref{eq:three_reactions_reversibiliy_condition}.

In any of the above cases, \Cref{prop:three_reactions_reversible_center} concludes the proof. \qed
\end{proof}

We depicted in \Cref{fig:three_reactions_reversible_center} some reaction networks that all fall under \Cref{cor:three_reactions_reversible_center}.

\begin{figure}[ht]
\begin{center}
\begin{tikzpicture}[scale=0.75]

\begin{scope}[shift={(0,0)}]
\draw [step=1, gray, very thin] (-1.25,-1.25) grid (3.25,3.25);
\draw [ ->, black] (-1.25,0)--(3.25,0);
\draw [ ->, black] (0,-1.25)--(0,3.25);
\node[inner sep=0,outer sep=1] (P1a) at (0,0)  {\large \textcolor{blue}{$\bullet$}};
\node[inner sep=0,outer sep=1] (P1b) at (1,-1)  {\large \textcolor{blue}{$\bullet$}};
\node[inner sep=0,outer sep=1] (P2a) at (2,1) {\large \textcolor{blue}{$\bullet$}};
\node[inner sep=0,outer sep=1] (P2b) at (2-1,1+2){\large \textcolor{blue}{$\bullet$}};
\node[inner sep=0,outer sep=1] (P3a) at (1,2) {\large \textcolor{blue}{$\bullet$}};
\node[inner sep=0,outer sep=1] (P3b) at (1-2,2+1){\large \textcolor{blue}{$\bullet$}};
\draw[arrows={-stealth},very thick,blue]  (P1a) to node {} (P1b);
\draw[arrows={-stealth},very thick,blue] (P2a) to node {} (P2b);
\draw[arrows={-stealth},very thick,blue] (P3a) to node {} (P3b);
\node [draw,fill=orange!20,very thick,rounded corners] (box) at (2.1,-0.1) {\scriptsize$\begin{array}{c}p>0\\c_2<0\\d_3>0\end{array}$};
\end{scope}

\begin{scope}[shift={(5.25,0)}]
\draw [step=1, gray, very thin] (-1.25,-1.25) grid (3.25,3.25);
\draw [ ->, black] (-1.25,0)--(3.25,0);
\draw [ ->, black] (0,-1.25)--(0,3.25);
\node[inner sep=0,outer sep=1] (P1a) at (0,0)  {\large \textcolor{blue}{$\bullet$}};
\node[inner sep=0,outer sep=1] (P1b) at (1,-1)  {\large \textcolor{blue}{$\bullet$}};
\node[inner sep=0,outer sep=1] (P2a) at (2,1) {\large \textcolor{blue}{$\bullet$}};
\node[inner sep=0,outer sep=1] (P2b) at (2,1+2){\large \textcolor{blue}{$\bullet$}};
\node[inner sep=0,outer sep=1] (P3a) at (1,2) {\large \textcolor{blue}{$\bullet$}};
\node[inner sep=0,outer sep=1] (P3b) at (1-2,2){\large \textcolor{blue}{$\bullet$}};
\draw[arrows={-stealth},very thick,blue]  (P1a) to node {} (P1b);
\draw[arrows={-stealth},very thick,blue] (P2a) to node {} (P2b);
\draw[arrows={-stealth},very thick,blue] (P3a) to node {} (P3b);
\node [draw,fill=orange!20,very thick,rounded corners] (box) at (2.1,-0.1) {\scriptsize$\begin{array}{c}p>0\\c_2=0\\d_3=0\end{array}$};
\end{scope}

\begin{scope}[shift={(10.5,0)}]
\draw [step=1, gray, very thin] (-1.25,-1.25) grid (3.25,3.25);
\draw [ ->, black] (-1.25,0)--(3.25,0);
\draw [ ->, black] (0,-1.25)--(0,3.25);
\node[inner sep=0,outer sep=1] (P1a) at (0, 0) {\large \textcolor{blue}{$\bullet$}};
\node[inner sep=0,outer sep=1] (P1b) at (1,-1)  {\large \textcolor{blue}{$\bullet$}};
\node[inner sep=0,outer sep=1] (P2a) at (2,1) {\large \textcolor{blue}{$\bullet$}};
\node[inner sep=0,outer sep=1] (P2b) at (2+1,1+2){\large \textcolor{blue}{$\bullet$}};
\node[inner sep=0,outer sep=1] (P3a) at (1,2) {\large \textcolor{blue}{$\bullet$}};
\node[inner sep=0,outer sep=1] (P3b) at (1-2,2-1){\large \textcolor{blue}{$\bullet$}};
\draw[arrows={-stealth},very thick,blue]  (P1a) to node {} (P1b);
\draw[arrows={-stealth},very thick,blue] (P2a) to node {} (P2b);
\draw[arrows={-stealth},very thick,blue] (P3a) to node {} (P3b);
\node [draw,fill=orange!20,very thick,rounded corners] (box) at (2.1,-0.1) {\scriptsize$\begin{array}{c}p>0\\c_2>0\\d_3<0\end{array}$};
\end{scope}

\begin{scope}[shift={(2,-5)}]
\draw [step=1, gray, very thin] (-3.25,-3.25) grid (1.25,3.25);
\draw [ ->, black] (-3.25,0)--(1.25,0);
\draw [ ->, black] (0,-3.25)--(0,3.25);
\node[inner sep=0,outer sep=1] (P1a) at (0,0)  {\large \textcolor{blue}{$\bullet$}};
\node[inner sep=0,outer sep=1] (P1b) at (1,-1)  {\large \textcolor{blue}{$\bullet$}};
\node[inner sep=0,outer sep=1] (P2a) at (-2,1) {\large \textcolor{blue}{$\bullet$}};
\node[inner sep=0,outer sep=1] (P2b) at (-2-1,1+2){\large \textcolor{blue}{$\bullet$}};
\node[inner sep=0,outer sep=1] (P3a) at (1,-2) {\large \textcolor{blue}{$\bullet$}};
\node[inner sep=0,outer sep=1] (P3b) at (1-2,-2+1){\large \textcolor{blue}{$\bullet$}};
\draw[arrows={-stealth},very thick,blue]  (P1a) to node {} (P1b);
\draw[arrows={-stealth},very thick,blue] (P2a) to node {} (P2b);
\draw[arrows={-stealth},very thick,blue] (P3a) to node {} (P3b);
\node [draw,fill=orange!20,very thick,rounded corners] (box) at (-2.1,-2.1) {\scriptsize$\begin{array}{c}p<0\\c_2<0\\d_3>0\end{array}$};
\end{scope}

\begin{scope}[shift={(7.25,-5)}]
\draw [step=1, gray, very thin] (-3.25,-3.25) grid (1.25,3.25);
\draw [ ->, black] (-3.25,0)--(1.25,0);
\draw [ ->, black] (0,-3.25)--(0,3.25);
\node[inner sep=0,outer sep=1] (P1a) at (0,0)  {\large \textcolor{blue}{$\bullet$}};
\node[inner sep=0,outer sep=1] (P1b) at (1,-1)  {\large \textcolor{blue}{$\bullet$}};
\node[inner sep=0,outer sep=1] (P2a) at (-2,1) {\large \textcolor{blue}{$\bullet$}};
\node[inner sep=0,outer sep=1] (P2b) at (-2,1+2){\large \textcolor{blue}{$\bullet$}};
\node[inner sep=0,outer sep=1] (P3a) at (1,-2) {\large \textcolor{blue}{$\bullet$}};
\node[inner sep=0,outer sep=1] (P3b) at (1-2,-2){\large \textcolor{blue}{$\bullet$}};
\draw[arrows={-stealth},very thick,blue]  (P1a) to node {} (P1b);
\draw[arrows={-stealth},very thick,blue] (P2a) to node {} (P2b);
\draw[arrows={-stealth},very thick,blue] (P3a) to node {} (P3b);
\node [draw,fill=orange!20,very thick,rounded corners] (box) at (-2.1,-2.1) {\scriptsize$\begin{array}{c}p<0\\c_2=0\\d_3=0\end{array}$};
\end{scope}

\begin{scope}[shift={(12.5,-5)}]
\draw [step=1, gray, very thin] (-3.25,-3.25) grid (1.25,3.25);
\draw [ ->, black] (-3.25,0)--(1.25,0);
\draw [ ->, black] (0,-3.25)--(0,3.25);
\node[inner sep=0,outer sep=1] (P1a) at (0,0)  {\large \textcolor{blue}{$\bullet$}};
\node[inner sep=0,outer sep=1] (P1b) at (1,-1)  {\large \textcolor{blue}{$\bullet$}};
\node[inner sep=0,outer sep=1] (P2a) at (-2,1) {\large \textcolor{blue}{$\bullet$}};
\node[inner sep=0,outer sep=1] (P2b) at (-2+1,1+2){\large \textcolor{blue}{$\bullet$}};
\node[inner sep=0,outer sep=1] (P3a) at (1,-2) {\large \textcolor{blue}{$\bullet$}};
\node[inner sep=0,outer sep=1] (P3b) at (1-2,-2-1){\large \textcolor{blue}{$\bullet$}};
\draw[arrows={-stealth},very thick,blue]  (P1a) to node {} (P1b);
\draw[arrows={-stealth},very thick,blue] (P2a) to node {} (P2b);
\draw[arrows={-stealth},very thick,blue] (P3a) to node {} (P3b);
\node [draw,fill=orange!20,very thick,rounded corners] (box) at (-2.1,-2.1) {\scriptsize$\begin{array}{c}p<0\\c_2>0\\d_3<0\end{array}$};
\end{scope}

\node[] at (6.25,6.2) {$c_1=1$, $d_1=-1$, $d_2=2$, $c_3=-2$};

\draw[ultra thin] (-1,5.85) -- (13.5,5.85);


\node[] at (1,5.5) {$c_2=-1$, $d_3=1$};

\node[] at (1,4.5) {$\begin{aligned} \dot{x} &= 3-x^p y^q -2x^q y^p \\ \dot{y} &= -3+2x^p y^q +x^q y^p \end{aligned}$};


\node[] at (6.25,5.5) {$c_2=0$, $d_3=0$};

\node[] at (6.25,4.5) {$\begin{aligned} \dot{x} &= 1-x^q y^p \\ \dot{y} &= -1 + x^p y^q \end{aligned}$};


\node[] at (11.5,5.5) {$c_2=1$, $d_3=-1$};

\node[] at (11.5,4.5) {$\begin{aligned} \dot{x} &= 1+x^p y^q -2x^q y^p \\ \dot{y} &= -1+2x^p y^q -x^q y^p \end{aligned}$};

\end{tikzpicture}
\end{center}
\caption{Some reaction networks that all fall under \Cref{cor:three_reactions_reversible_center}, along with the differential equation \eqref{eq:three_reactions_ode_scaled_reversible}. Among the graphs, in the top row we have $p>0$ and $-p<q<p$, while in the bottom row we have $p<0$ and $p<q<-p$.}
\label{fig:three_reactions_reversible_center}
\end{figure}
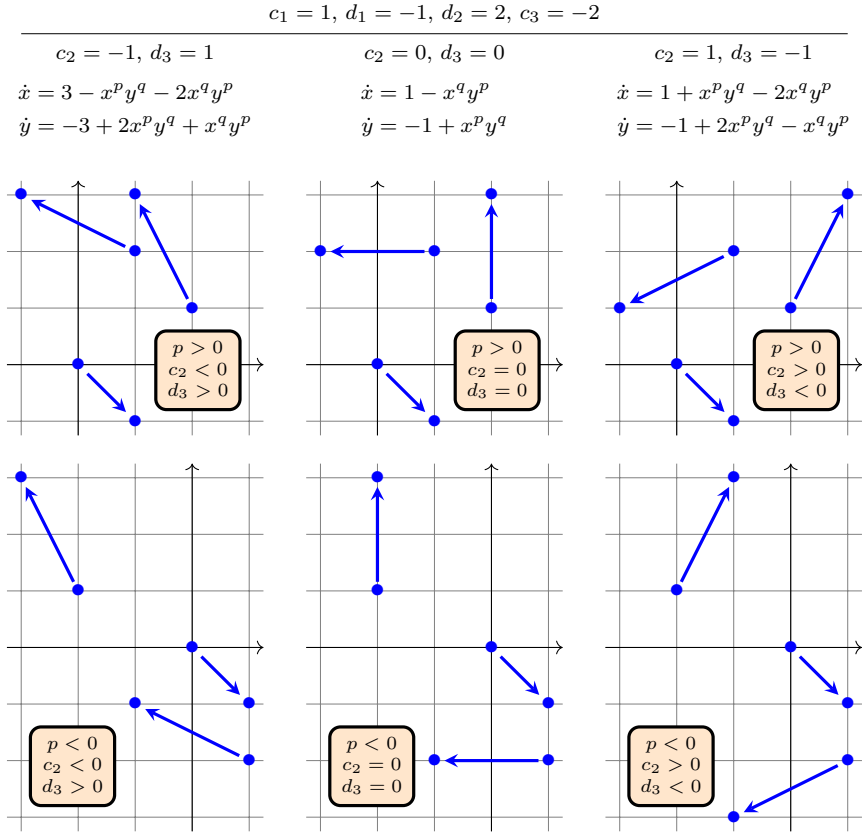

Now fix $p$, $q$, $c_1$, $c_2$, $c_3$, $d_1$, $d_2$, $d_3$ such that all the assumptions of \Cref{cor:three_reactions_reversible_center} are fulfilled, and consider the mass-action system \eqref{eq:three_reactions_ode} with \eqref{eq:three_reactions_reversible_ab_pq}. How to choose $\kappa_1$, $\kappa_2$, $\kappa_3$ in order that the unique positive equilibrium is a center? First note that there exists an $(\overline{x},\overline{y})\in\mathbb{R}^2_+$ for which \eqref{eq:three_reactions_kappa_kappabar} with \eqref{eq:three_reactions_reversible_ab_pq} holds if and only if
\begin{align*}
\kappa_1 = \overline{\kappa}_1 \text{ and } \frac{\kappa_3}{\kappa_2}=\frac{\overline{\kappa}_3}{\overline{\kappa}_2}K^{p-q}.
\end{align*}
By \eqref{eq:three_reactions_reversibiliy_condition}, $\overline{\kappa}_1$ is arbitrary and $\frac{\overline{\kappa}_3}{\overline{\kappa}_2}K^{p-q}=-\frac{c_2}{d_3}\left(-\frac{c_1}{d_1}\right)^{p-q-1}$. Thus, the answer to the above question is that one has to choose $\kappa_1$, $\kappa_2$, $\kappa_3$ such that
\begin{align*}
\kappa_1 \text{ is arbitrary and } \frac{\kappa_3}{\kappa_2}=-\frac{c_2}{d_3}\left(-\frac{c_1}{d_1}\right)^{p-q-1}.
\end{align*}

Finally, we remark that the condition $\left\lvert\frac{d_1}{c_1}\right\rvert = \sqrt{\left\lvert\frac{d_2}{c_2}\right\rvert \left\lvert\frac{d_3}{c_3}\right\rvert}$ in \Cref{cor:three_reactions_reversible_center} sheds light on why we had to set $(a_4,b_4)=\left(q-p,p+\frac{q^2}{p}\right)$ in \Cref{subsec:chain_reversible_center}. With this choice, the absolute values of the slopes of the three reactions are $\left\lvert\frac{q}{p}\right\rvert$, $1$, $\frac{q^2}{p^2}$, the first one being the geometric mean of the latter two.

\subsection{Li\'enard center}
\label{subsec:three_reactions_lienard_center}

Let us consider now the mass-action system \eqref{eq:three_reactions_network} with 
\begin{align*}
(a_1,b_1)=(1,0), (a_2,b_2)=\left(0,-\frac{1}{2}\right), (a_3,b_3)=(0,-2).
\end{align*}
Its associated scaled differential equation is then
\begin{align}
\label{eq:three_reactions_ode_scaled_lienard}
\begin{split}
\dot{x} &= c_1\overline{\kappa}_1 x + c_2\overline{\kappa}_2 y^{-\frac{1}{2}} + c_3\overline{\kappa}_3 y^{-2}, \\
\dot{y} &= K(d_1\overline{\kappa}_1 x + d_2\overline{\kappa}_2  y^{-\frac{1}{2}}+ d_3\overline{\kappa}_3 y^{-2}).
\end{split}
\end{align}

\begin{proposition}
\label{prop:three_reactions_lienard_center}
Consider the differential equation \eqref{eq:three_reactions_ode_scaled_lienard} with \eqref{eq:three_reactions_kappasubst}. Assume that $\frac{K}{\lambda}>0$ and
\begin{align}
\label{eq:three_reactions_lienard_condition}
c_1\overline{\kappa}_1=K d_2 \overline{\kappa}_2=4K d_3 \overline{\kappa}_3\neq0
\end{align}
hold. Then the equilibrium $(1,1)$ is a center.
\end{proposition}
\begin{proof}
The determinant and the trace of the Jacobian matrix at $(1,1)$ are
\begin{align*}
\frac{3}{2\lambda}K \overline{\kappa}_1\overline{\kappa}_2\overline{\kappa}_3\text{ and }
c_1\overline{\kappa}_1-\frac{1}{2}K d_2 \overline{\kappa}_2 -2 K d_3 \overline{\kappa}_3,
\end{align*}
respectively. By the assumptions, the former is positive, the latter is zero, and therefore, the eigenvalues are purely imaginary.

Shifting the equilibrium to the origin yields
\begin{align}
\label{eq:three_reactions_ode_scaled_lienard2}
\begin{split}
\dot{x} &= c_1\overline{\kappa}_1 (x+1) + c_2\overline{\kappa}_2 (y+1)^{-\frac{1}{2}} + c_3\overline{\kappa}_3 (y+1)^{-2}, \\
\dot{y} &= K[d_1\overline{\kappa}_1 (x+1) + d_2\overline{\kappa}_2  (y+1)^{-\frac{1}{2}}+ d_3\overline{\kappa}_3(y+1)^{-2}].
\end{split}
\end{align}
Differentiation of the second equation w.r.t. time and then application of each of the two equations once yields that \eqref{eq:three_reactions_ode_scaled_lienard2} is equivalent to the Li\'enard equation
\begin{align}
\label{eq:three_reactions_lienard_yddot}
\ddot{y} + f(y) \dot{y} + g(y) = 0,
\end{align}
where
\begin{align*}
f(y) &= - c_1\overline{\kappa}_1 + \frac{1}{2}Kd_2\overline{\kappa}_2  (y+1)^{-\frac{3}{2}} + 2Kd_3\overline{\kappa}_3  (y+1)^{-3},\\
g(y) &= \frac{1}{\lambda}K\overline{\kappa}_1\overline{\kappa}_2\overline{\kappa}_3\left[(y+1)^{-\frac{1}{2}}-(y+1)^{-2}\right].
\end{align*}
By \cite[Theorem 4.1]{christopher:li:2007}, the origin is a center for \eqref{eq:three_reactions_lienard_yddot} if and only if $F=\Phi\circ G$ for some analytic function $\Phi$ with $\Phi(0)=0$, where $F(x)=\int_0^x f(y)\mathrm{d}y$ and $G(x)=\int_0^x g(y)\mathrm{d}y$. Now
\begin{align*}
F(x) &= -c_1\overline{\kappa}_1 x - K d_2 \overline{\kappa}_2 \left[(x+1)^{-\frac{1}{2}}-1\right] - K d_3 \overline{\kappa}_3 \left[(x+1)^{-2}-1\right],\\
G(x) &= \frac{1}{\lambda}K\overline{\kappa}_1\overline{\kappa}_2\overline{\kappa}_3\left[2(x+1)^{\frac{1}{2}}+(x+1)^{-1}-3\right].
\end{align*}
A short calculation shows that under the hypothesis \eqref{eq:three_reactions_lienard_condition}, $F=\Phi\circ G$ holds with $\Phi(z)=\alpha z^2 +\beta z$, where
\begin{align*}
\alpha=-\frac{\lambda^2}{4}\frac{c_1\overline{\kappa}_1}{(K\overline{\kappa}_1\overline{\kappa}_2\overline{\kappa}_3)^2} \text{ and } \beta=-\frac{3\lambda}{2}\frac{c_1\overline{\kappa}_1}{K\overline{\kappa}_1\overline{\kappa}_2\overline{\kappa}_3}.
\end{align*}
This concludes the proof. \qed
\end{proof}

\begin{figure}[b]
\begin{center}
\input{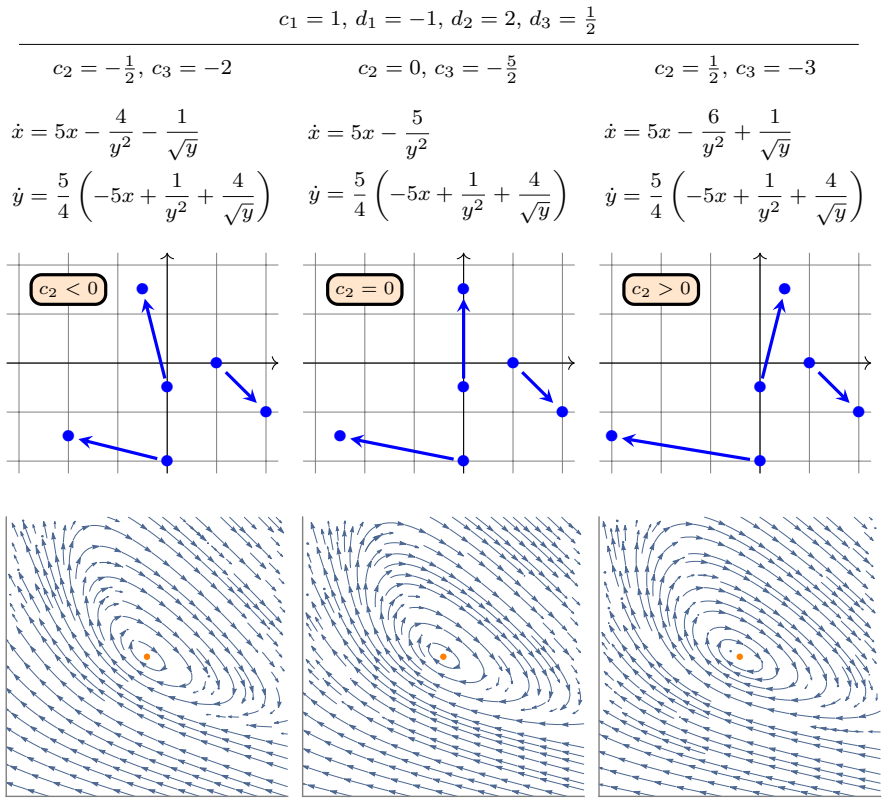}
\end{center}
\caption{Some reaction networks that all fall under \Cref{cor:three_reactions_lienard_center}, along with the differential equation \eqref{eq:three_reactions_ode_scaled_lienard}, and the corresponding phase portraits.}
\label{fig:three_reactions_lienard}
\end{figure}

\begin{corollary}
\label{cor:three_reactions_lienard_center}
Consider the differential equation \eqref{eq:three_reactions_ode_scaled_lienard} with \eqref{eq:three_reactions_kappasubst}. Assume that $K>0$, $\sgn c_1 = -\sgn d_1 = \sgn d_2 = \sgn d_3 \neq0$, and
\begin{align*}
\frac{c_3}{d_3}<\frac{c_1}{d_1}<\frac{c_2}{d_2} \text{ and } \frac{c_1}{d_1}=-\frac{4}{5}K=\frac{4}{5}\frac{c_2}{d_2}+\frac{1}{5}\frac{c_3}{d_3}
\end{align*}
hold. Then the equilibrium $(1,1)$ is a center.
\end{corollary}
\begin{proof}
By the assumptions, each of $c_2d_3-c_3d_2$, $c_3d_1-c_1d_3$, $c_1d_2-c_2d_1$ is positive, and hence $\lambda>0$. Thus, $\frac{K}{\lambda}$ is also positive. Further, under the assumptions of this corollary, it is straightforward to verify the condition \eqref{eq:three_reactions_lienard_condition}. \Cref{prop:three_reactions_lienard_center} then concludes the proof. \qed
\end{proof}

We depicted in \Cref{fig:three_reactions_lienard} some reaction networks that all fall under \Cref{cor:three_reactions_lienard_center}.

\section{Zigzag}
\label{sec:zigzag}

We conclude with an example of a reaction network, where the existence of a positive equilibrium does depend on the choice of the rate constants, a phenomenon that appeared for neither of the networks in \Cref{sec:quadrangle,sec:chain,sec:three_reactions}.

The mass-action system under investigation in this section, along with its associated differential equation takes the form
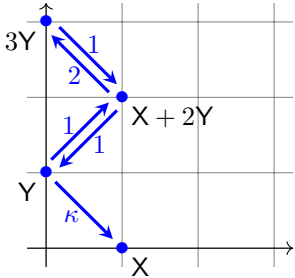
\begin{center}
\begin{tikzpicture}

\draw [step=1, gray, very thin] (-0.25,-0.25) grid (3.25,3.25);
\draw [ ->, black] (-0.25,0)--(3.25,0);
\draw [ ->, black] (0,-0.25)--(0,3.25);

\node[inner sep=0,outer sep=1] (P1) at (0,3) {\large \textcolor{blue}{$\bullet$}};
\node[inner sep=0,outer sep=1] (P2) at (1,2) {\large \textcolor{blue}{$\bullet$}};
\node[inner sep=0,outer sep=1] (P3) at (0,1) {\large \textcolor{blue}{$\bullet$}};
\node[inner sep=0,outer sep=1] (P4) at (1,0) {\large \textcolor{blue}{$\bullet$}};

\node [below left]  at (P1) {$3\mathsf{Y}$};
\node [below right] at (P2) {$\mathsf{X}+2\mathsf{Y}$};
\node [below left]  at (P3) {$\mathsf{Y}$};
\node [below right] at (P4) {$\mathsf{X}$};

\draw[arrows={-stealth},very thick,blue,right,pos=0.3,transform canvas={xshift=1.5pt, yshift=1.5pt}]  (P1) to node {$1$} (P2);
\draw[arrows={-stealth},very thick,blue,left,pos=0.3,transform canvas={xshift=-1.5pt, yshift=-1.5pt}]  (P2) to node {$2$} (P1);
\draw[arrows={-stealth},very thick,blue,right,pos=0.6,transform canvas={xshift=1.5pt, yshift=-1.5pt}]       (P2) to node {$1$} (P3);
\draw[arrows={-stealth},very thick,blue,left,pos=0.6,transform canvas={xshift=-1.5pt, yshift=1.5pt}] (P3) to node {$1$} (P2);
\draw[arrows={-stealth},very thick,blue,left,pos=0.6]        (P3) to node {$\kappa$} (P4);

\node [] at (4,3/2) {and};

\node [] at (7,3/2) {$\begin{aligned}
\dot{x} &= y^3 - 3xy^2+(1+\kappa)y, \\
\dot{y} &= -y^3 + xy^2+(1-\kappa)y.
\end{aligned}$};

\end{tikzpicture}
\end{center}
There is a unique positive equilibrium at $\left(\frac{1}{\sqrt{2-\kappa}}, \sqrt{2-\kappa}\right)$ for $\kappa < 2$ and no positive equilibrium for $\kappa \geq 2$. The determinant of the Jacobian matrix at the equilibrium is positive, while its trace is $5\kappa-9$, which becomes positive for $\kappa > \frac95$. The Andronov--Hopf bifurcation at $\kappa = \frac95$ is subcritical, since $L_1 = \frac{5\pi}{13} > 0$.

The $x$--axis is invariant, consists of equilibria, and attracts nearby points from ${\mathbb R}^2_+$ if $\kappa > 1$. For $\kappa > \frac95$ it seems that all orbits except the positive equilibrium converge to the $x$--axis.


\newpage
\bibliographystyle{spmpsci}      
\bibliography{biblio}   

\end{document}